\documentclass[10pt]{article}
\usepackage{amsmath}
\usepackage{amssymb}
\usepackage{amsthm}
\usepackage{graphicx}
\usepackage{verbatim}
\usepackage{enumerate}
\usepackage{color}
\usepackage{mathrsfs}
\usepackage{longtable}
\numberwithin{equation}{section}

\theoremstyle{plain}
\newtheorem{thm}{Theorem}[section]

\newtheorem{lem}[thm]{Lemma}
\newtheorem{pr}[thm]{Proposition}
\newtheorem{cor}[thm]{Corollary}
\newtheorem{defn}[thm]{Definition}
\newtheorem{outline}[thm]{Outline}

\theoremstyle{remark}

\allowdisplaybreaks

\def\N{\mathbb{N}}

\def\QQ{{\bf Q}}
\def\R{\mathbb{R}}

\def\CC{{\cal C}}

\def\Var{{\rm Var} \,}
\def\ee{\varepsilon}
\def\E{{\bf E}}
\def\P{{\bf P}}
\def\Cox{\hfill \Box}

\def\disp{\displaystyle}
\def\one{{\bf 1}}
\def\|{{\, | \, }}
\def\F{{\mathcal F}}
\def\T{{\mathcal T}}
\def\B{{\mathcal B}}
\def\BB{{\mathcal B}}
\def\CC{{\mathcal C}}

\def\rt{{\bf 0}}

\def\ulam{{\mathcal U}}
\def\VU{{\bf V}}  
\def\bfa{{\bf a}}

\def\conn{\leftrightarrow}
\def\FO{\mathrm{F.O.}}
\def\SO{\mathrm{S.O.}}
\def\G{{\mathcal G}}
\def\pivot{\beta}

\def\ul{\underline}
\def\GW{{\tt GW}}
\def\Bin{\text{Bin}}
\def\one{\mathbf{1}}
\def\Mbar{\overline{M}}
\def\rtt{{\mathbf{0}}}
\def\TT{{\bf T}}

\makeatletter
\DeclareRobustCommand{\Par}[1]{%
  \mathpalette\do@Par{#1}%
}
\newcommand{\do@Par}[2]{%
  \fix@Par{#1}{+}%
  \reflectbox{$\m@th#1\vec{\reflectbox{$\fix@Par{#1}{-}\m@th#1#2\fix@Par{#1}{+}$}}$}%
  \fix@Par{#1}{-}%
}
\newcommand{\fix@Par}[2]{%
  \ifx#1\displaystyle
    \mkern#23mu
  \else
    \ifx#1\textstyle
      \mkern#23mu
    \else
      \ifx#1\scriptstyle
        \mkern#22mu
      \else
        \mkern#22mu
      \fi
    \fi
  \fi
}
\makeatother

\def\qt{{\tilde{q}}}

\def\cL{{\mathcal L}}
\def\dual{\beta^*}
\def\KL{{\mathcal K}}
\def\df{\dotfill}

\begin{document}

\begin{center}
{\large \bf Invasion Percolation on Galton-Watson Trees}
\end{center}

\begin{flushright}
Marcus Michelen, Robin Pemantle\footnote{Partially supported by NSF grant DMS-1612674} and Josh Rosenberg \\
{\tt \{marcusmi, pemantle, rjos\}@math.upenn.edu}
\end{flushright}
\begin{abstract}
    We consider invasion percolation on Galton-Watson trees.  On almost every Galton-Watson tree, the invasion cluster almost surely contains only one infinite path.  This means that for almost every Galton-Watson tree, invasion percolation induces a probability measure on infinite paths from the root.  We show that under certain conditions of the progeny distribution, this measure is absolutely continuous with respect to the limit uniform measure.  This confirms that invasion percolation, an efficient self-tuning algorithm, may be used to sample approximately from the limit uniform distribution. Additionally, we analyze the forward maximal weights along the backbone of the invasion cluster and prove a limit law for the process.
\end{abstract}

{{\bf Keywords}:  Backbone, incipient infinite cluster, limit uniform, Poisson point process, pivot, self-organized criticality.}

\clearpage

\setcounter{section}{0}
\section{Introduction} \label{sec:intro}

Given an infinite rooted tree, how might one sample, nearly uniformly,
from the set of paths from the root to infinity?  One motive for this 
question is that nearly uniform sampling leads to good estimates on 
the growth rate~\cite{jerrum-sinclair}.  One might be trying to estimate the
size of a search tree, or, in the case of~\cite{randall-sinclair},
to determine the growth rate of the number of self-avoiding paths.

A number of methods have been studied.  One is to do a random walk on
the tree, with a ``homesickness'' parameter determining how much steps
back toward the root are favored~\cite{LPP-biased}.  The parameter needs
to be tuned near criticality: too much homesickness and the walk gets
stuck near the root; too little homesickness and the walk goes to infinity
without taking the time to ensure that the path is well randomized.
Randall and Sinclair~\cite{randall-sinclair} solve this by estimating
the critical parameter as the walk progresses, re-tuning the
homesickness to lie above this by an amount decreasing at an 
appropriate rate.  

Another approach is to use percolation.  One conditions the percolation 
cluster to survive to level $N$; as the percolation parameter decreases to 
criticality and $N$ is taken to infinity, the law of this
cluster approaches the law of the {\em incipient infinite cluster} (IIC).  For 
many graphs---e.g. regular or Galton-Watson trees---the IIC almost surely contains
a unique infinite path, thereby giving a mechanism for sampling such a path.
In practice, the same considerations arise as with homesick random walks:
tuning the percolation parameter too low yields too little likelihood of
survival and too great a time cost to rejection sampling; too great
a percolation parameter results in too many surviving paths and a 
selection problem which leads to poor randomization.  

Invasion percolation was introduced as a model for how viscous fluid creeps 
through an environment in~\cite{invasion}.   
Each site is given an independent $U[0,1]$ random variable, 
representing how great the percolation probability would have 
to be before the site would be open.  The cluster then grows 
by adding, at each time step, the site with the least $U$ value 
among sites neighboring the cluster but not in the cluster.  It is 
not hard to see that the $\limsup$ of $U$-values of bonds chosen is equal 
to the critical percolation parameter.  In other words, instead of running 
percolation at $p_c$ and conditioning to survive, one allows slightly 
supercritical bonds but less and less as the cluster grows.  As is the case for 
the IIC, the invasion cluster almost surely contains only one infinite path in 
the case of regular or Galton-Watson trees, and thus gives a different mechanism for sampling paths.  
Unlike the IIC and homesick random walk, invasion percolation requires no tuning to criticality and 
is an instance of self-organized criticality.

The invasion cluster has some properties in common with the IIC but not all.  For 
example, results of Kesten~\cite{kesten-IIC} and Zhang~\cite{zhang-invasion} 
show that the growth exponents of the two are equal on the two-dimensional 
lattice; however the measures of the two clusters are mutually singular 
on the lattice~\cite{damron-invasion} as well as on a regular 
tree~\cite{angel-invasion2008}.  Our focus is the comparison of the laws induced 
on paths by both the IIC and invasion percolation.

On a Galton-Watson tree $T$, there is a natural measure on paths, the 
limit-uniform measure $\mu_T$, which although it does not restrict precisely
to the uniform measure on each generation, approximates this as closely
as possible.  There is not, however, a fast algorithm for sampling from it.  
Rules such as ``split equally at each node'' lead to rapid sampling but the
wrong entropy; in other words, the Radon-Nikodym derivative with respect
$\mu_T$ on generation $N$ will be exponential in $N$.  It is not hard
to show that on almost every Galton-Watson tree (assuming a $Z \log Z$
moment for the offspring distribution), the unique path in the IIC 
has law $\mu$.  Since sampling
from the IIC is problematic, it is therefore natural to ask how close
the law $\nu_T$ of the path chosen by the invasion cluster is to $\mu_T$.
Showing that it has the right entropy is not too involved.  This is
Theorem~\ref{th:small} below.  It is also easy to see that the two
laws are typically not equal.  As an example, consider the set of trees with first three generations given by 

\begin{figure}[h]   \centering
    \includegraphics{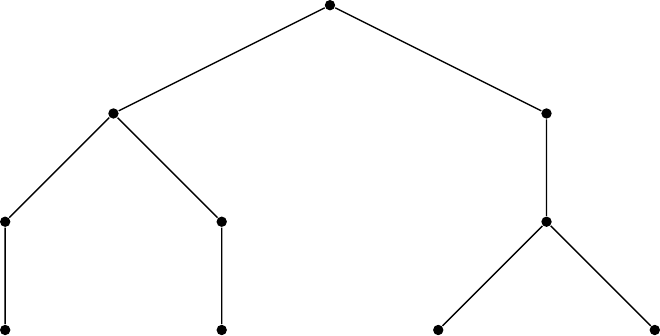}
\end{figure}

When averaged over the remaining generations with the Galton-Watson measure---or equivalently, placing independent Galton-Watson trees at the terminal nodes---the limit-uniform measure splits equally at the root, while the invasion measure favors the left subtree regardless of the offspring distribution.

The best comparison one might hope for is that 
$\nu_T$ be absolutely continuous with respect to $\mu_T$, perhaps
even with Radon-Nikodym derivative in $L^p$.  Our main result is
as follows.

\begin{thm} \label{th:main}
Suppose the offspring distribution $Z$ has at least $p$ moments, set 
$p_1 := \P[Z = 1]$, let $\mu:= \E[Z]$, and denote $q:=\frac{\log\mu}{\log\left(1/p_1\right)}$.  If 
$$2p^2 q^2+(3p^2+5p)q+(-p^2+11p-4)<0, $$ 
then $\nu_\TT \ll \mu_\TT$ almost surely.  
\end{thm}

The condition in Theorem \ref{th:main} is a trade-off between $p_1$ and $p$.  In the case of
$p = \infty$, the condition becomes $p_1 < 1/\mu^{\frac{3+\sqrt{17}}{2}}$.  In the case of $p_1 = 0$, 
the condition is $p>\frac{11+\sqrt{105}}{2}$.

An outline of the argument behind Theorem~\ref{th:main} is as follows.  
Let $X_n$ be the KL-distance between the way that $\mu_T$ and $\nu_T$  
split at the $n$th step $\gamma_n$ of a path chosen from $\nu_T$.
A sufficient condition for absolute continuity is that 
$\sum_{n=1}^\infty \E X_n < \infty$.  A precise statement is given
in Lemma~\ref{lem:X} below.  A more detailed outline of why this condition should
hold is given at the beginning of Section~\ref{sec:rest}.

The reason we have a hope of estimating $X_n$ is that there is a 
{\em backbone} decomposition for invasion percolation.  Define
the backbone to be the almost surely unique nonbactracking path
$\gamma = (\rtt, \gamma_1, \gamma_2, \ldots)$ from the root to infinity.
For any vertex $v$ define the {\em pivot value} at $v$, denoted 
$\pivot (v)$ \label{pivot-def}, to be the least $p$ such that there is a path from 
$v$ to infinity in the subtree at $v$ with all $U$ variables 
(not including the one at $v$) at most $p$.
On a regular tree, invasion percolation was studied 
in~\cite{angel-invasion2008,angel-invasion2013}.
For the purposes of studying $\nu_T$, the regular tree is a degenerate 
case, because $\mu_T$ and $\nu_T$ are equal to each other and to the 
equally splitting measure.  However their results on backbones and
pivots extend in a useful way to the Galton-Watson setting.

Conditioning on $T$, the way the invasion measure splits at $v$
depends on the whole tree.  However, if one also conditions on the
pivot at $v$, then the way the invasion
measure splits at $v$ becomes independent of everything outside of the
subtree at $v$.  A similar statement is true if one conditions on the 
pivot of $v$ being less than or equal to a certain value; these are 
the Markov properties of Propositions \ref{pr:markov 2} and \ref{pr:markov}.  The limiting behavior of these values are given in Theorem \ref{theorem:condist} and Corollary \ref{cor:LPE}.
Further, Lemma \ref{lem:qt-ratio-survival} shows that this 
conditioned splitting measure is close to a ratio of survival probabilities under
supercritical Bernoulli percolation.  The problem is thus reduced to proving estimates
of the survival probabilities of Galton-Watson trees under supercritical Bernoulli
percolation as in Section \ref{sec:bounds-E}.

The remainder of the paper is organized as follows.  
Section~\ref{sec:prelim} sets up the notation and gives some 
preliminary results.  Some care is required to set up the probability
space so that we can easily speak of the random measures $\mu_T$ and 
$\nu_T$, which are conditional on the Galton-Watson tree.  
Section~\ref{sec:prelim} culminates in Lemma~\ref{lem:X} and Corollary~\ref{cor:X}. 
Section~\ref{sec:g} gives an upper bound on survival probabilities and squeezes survival
probabilities close to an easier-to-analyze averaged version with high probability. 
Section~\ref{sec:backbone} proves two Markov properties for the subtree 
from $\gamma_n$ together with $\pivot (\gamma_n)$.  The remainder of 
the section extends the work of~\cite{angel-invasion2008} by proving 
a limit law for $\pivot (\gamma_n)$ which then implies an upper bound
on the rate at which $\pivot (\gamma_n) \downarrow p_c$.  In particular, Corollary~\ref{cor:LPE} shows convergence to the Poisson lower envelope process, as in~\cite{angel-invasion2008}.
Section~\ref{sec:rest} proves Theorem~\ref{th:main} by comparing the conditional invasion measure
to the ratio of survival probabilities and by carefully estimating survival probabilities near criticality.

A glossary of notation by page of reference is included after the references.

\section{Construction and preliminary results} \label{sec:prelim}

\subsection{Galton-Watson trees}

We begin with some notation we use for all trees, random or not.
Let $\ulam$ \label{ulam-def} be the canonical {\em Ulam-Harris} tree~\cite{louigi-ford}.
The vertex set of $\ulam$ is the set $\VU : = \bigcup_{n=1}^\infty \N^n$, \label{VU-def} 
with the empty sequence $\rtt := \emptyset$ \label{rtt-def} as the root.  There is an 
edge from any sequence $\bfa = (a_1, \ldots , a_n)$ to any extension 
$\bfa \sqcup j := (a_1, \ldots , a_n , j)$.  The depth of a vertex $v$
is the graph distance between $v$ and $\rt$ and is denoted $|v|$.\label{depth-def}
We work with trees $T$ that are locally finite rooted subtrees of $\ulam$.
The usual notations are in force: $T_n$ \label{T_n-def} denotes the set of vertices 
at depth $n$; $T(v)$ \label{T(v)-def} is the subtree of $T$ at $v$, canonically identified
with a rooted subtree of $\ulam$, in other words the vertex set of 
$T(v)$ is $\{ w : v \sqcup w \in V(T) \}$ \label{sqcup-def}; 
$\partial T$ \label{partial-T-def} denotes the set of infinite non-backtracking paths from 
the root; if $\gamma \in \partial T$ then \label{gamma_n-def} $\gamma_n$ $(n \geq 0)$ 
denotes the $n$th vertex in $\gamma$; the last common ancestor of 
$v$ and $w$ is denoted $v \wedge w$ \label{LCA-vertex-def} and the last common vertex of 
$\gamma$ and $\gamma'$ is denoted $\gamma \wedge \gamma'$ \label{LCA-path-def}; $\Par{v}$ denotes the parent
of $v$ \label{Par-def}.  
Let $\mu_T^n$ \label{mu_T^n-def} denote the uniform measure on the $n$th generation of $T$.
In some cases, for example for almost every Galton-Watson tree,
the limit $\mu_T := \lim_{n \to \infty} \mu_T^n$ \label{mu-T-def} exists and is
called the {\em limit-uniform measure}.

Turning now to Galton-Watson trees,
let $\phi (z) := \sum_{n=1}^\infty p_n z^n$ \label{phi-def} be the ordinary generating function 
for a supercritical branching process with no death, i.e., $\phi (0) = 0$.  
We recall,
\begin{eqnarray*}
\phi'  (1) & = & \E Z =: \mu \\ \label{mu-def}
\phi'' (1) & = & \E [Z (Z-1)] 
\end{eqnarray*}
where $Z$ \label{Z-def} is a random variable with probability generating function $\phi$.  Throughout, we assume $\E[Z^2] < \infty$; in particular, this also means that $\phi''(1) < \infty$.  Moreover, since our focus
is on $\partial T$, the assumption of $\phi(0) = 0$ can be made without loss of generality by considering the reduced tree, as in \cite[Chapter I.12]{athreya-ney}.

We will work on the canonical probability space $(\Omega , \F , \P)$ \label{space-def}
where $\Omega = (\N \times [0,1])^{\VU}$, $\F$ is the product Borel
$\sigma$-field, and $\P$ is the probability measure making the
coordinate functions $\omega_v = (\deg_v , U_v)$ IID with the law
of $(Z,U)$, where $U$ is uniform on $[0,1]$ \label{U_v-def} and independent of $Z$. 
The variables $\{ \deg_v \}$---where $\deg_v$ \label{deg-def} is interpreted as the number of
children of vertex $v$---will construct the Galton-Watson tree, 
while the variables $\{ U_v \}$ will be used later for percolation.  
Let $\TT$\label{TT-def} be the random rooted subtree of $\ulam$ which is the 
connected component containing the root of the set of vertices 
that are either the root or are of the form $v \sqcup j$ such that 
$0 \leq j < \deg_v$.  This is a Galton-Watson tree with ordinary generating 
function $\phi$.


As is usual for Galton-Watson branching processes, we denote
$Z_n := |\TT_n|$. \label{Z_n-def}  Extend this by letting $Z_n (v)$ \label{Z_n(v)-def} denote the
number of offspring of $v$ in generation $|v| + n$; similarly, extend the
notation for the usual martingale \label{W_n-def} $W_n := \mu^{-n} Z_n$ by letting
$W_n (v) := \mu^{-n} Z_n (v)$. \label{W_n(v)-def}  We know that $W_n (v) \to W(v)$ \label{W-def}
for all $v$, almost surely and in $L^p$ if the offspring distribution
has $p$ moments.   This is stated without
proof for integer values of $q \geq 2$ in~\cite[p.~16]{harris-BP} 
and \cite[p.~33, Remark~3]{athreya-ney}; for a proof for all $q > 1$,
see~\cite[Theorems~0~and~5]{bingham-doney74}.  Further extend this notation by letting $v^{(i)}$
denote the $i$th child of $v$, letting $Z_n^{(i)} (v)$ \label{Z_n^(i)-def} denote
$n$th generation descendants of $v$ whose ancestral line passes
through $v^{(i)}$, and letting \label{W_n^(i)(v)-def} $W_n^{(i)} (v) := \mu^{-n} Z_n^{(i)} (v)$.
Thus, for every $v$, $W(v) = \sum_i W^{(i)} (v)$.  For convenience, define $p_c := 1/\mu$ \label{p_c-def} and recall that $p_c$ is almost surely the critical percolation parameter for $\TT$ \cite{lyons90}.

\subsection{Bernoulli and Invasion Percolation}

In this subsection we give the formal construction of percolation
on random trees, and for invasion percolation.  Our approach is
to define a simultaneous coupling of invasion percolations on 
all subtrees $T$ of $\ulam$ via the $U$ variables, then specialize 
to the random tree $\TT$.  Let \label{T-def} $\T := \sigma (\{ \deg_v : v \in \VU \})$
denote the $\sigma$-field generated by the tree variables.  
Because $\T$ is independent from the $U$ variables, this means 
we have constructed a process whose law, conditional on $\T$, is 
invasion percolation on $\TT$.
We use the notation $\E_*$\label{E_*-def} to denote $\E [\cdot \| \T]$;
similarly $\P_* [\cdot] := \P[\cdot \| \T]$ \label{P_*-def}.  Moreover, we use  
$\GW:= \P|_{\T}$ \label{GW-def} to denote the Galton-Watson measure on trees.

We begin with a similar construction for ordinary percolation.
For $0 < p < 1$, simultaneously define Bernoulli$(p)$ percolations
on rooted subtrees $T$ of $\ulam$ by taking the percolation clusters 
to be the connected component containing $\rtt$ of the induced subtrees
of $T$ on all vertices $v$ such that $U_v \leq p$.
Let $\F_n$ \label{F_n-def} be the $\sigma$-field generated by the variables 
$\{ U_v , \deg_v : |v| < n \}$.  Let $p_c = 1/\mu = 1/\phi'(1)$ 
denote the critical probability for percolation.  Write
$v \conn_{T,p} w$ \label{conn-def} if $U_u \leq p$ for all $u$ on the geodesic
from $v$ to $w$ in $T$.  Informally, $v \conn_{T,p} w$ iff $v$ and $w$ 
are both in $T$ and are connected in the $p$-percolation.
The event of successful $p$ percolation on $T$ is 
$H_T (p) := \{ \rtt \conn_{T,p} \infty \}$ and the event of 
successful $p$ percolation on the random tree $\TT$,
is denoted $H_\TT (p)$ or simply\label{H(p)-def} $H(p)$.  Let \label{g(T,p)-def} $g(T , p) := 
\P [H_T (p)]$ denote the probability of $p$ percolation
on $T$.  The conditional probability $\P_* [H(p)]$ is measurable 
with respect to $\TT$ and we may define $g(\TT , p):= \P_* [H(p)]$.  Furthermore, we may define \label{g(p)-def} $g(p) = \P [H(p)] = \E g(\TT , p)$.  Since $p_c = 1/\mu$ is the critical percolation parameter for a.e. $\TT$, note that $g(\TT,p) = 0$ for all $p \in [0,p_c]$.

Before defining invasion percolation, we record some basic properties of $g$.

\begin{pr} \label{pr:K}
The derivative from the right \label{K-def} $K := g'(p_c)$ exists and is given by 
\begin{equation} \label{eq:K}
K := \frac{2}{p_c^3 \phi'' (1)} \, .
\end{equation}
\end{pr} 

\begin{proof} Let $\phi_p (z) := \phi (1 - p
+ p z)$ be the offspring generating function for the 
Galton-Watson tree thinned by $p$-percolation for $p \in (p_c,1)$.  
The fixed point
of $\phi_p$ is $1 - g(p)$.  In other words, $g(p)$ is the
unique $s \in (0,1)$ for which $1 - \phi_p (1-s) = s$.
The first two derivatives of $\phi_p$ at $1$ are given by
\begin{eqnarray*}
\phi_p'(1) & = & \frac{p}{p_c} \, ; \\[1ex]
\phi_p''(1) & = & p^2 \phi'' (1) \, . \\[1ex]
\end{eqnarray*}
By a Taylor expansion, this leads to 
$$1 - \phi_p (1-s) = \frac{p}{p_c} s - \frac{1+o(1)}{2}
   \phi'' (1) p^2 s^2$$
as $p \downarrow p_c$.  Setting this equal to $s$ and solving for 
$s \in (0,1)$ yields the conclusion of the proposition.
$\Cox$
\end{proof}

\begin{cor}\label{cor:dersurvprob}As $p \downarrow p_c$,
$g'(p) \to K$.
\end{cor}
\begin{proof}
The existence of $g'(p)$ on $(p_c, 1)$ follows from the 
implicit function theorem.  To obtain an expression for $g'(p)$, 
we differentiate both sides of the expression 
$\phi(1 - p \cdot g(p)) = 1 - g(p)$ 
with respect to $p$, which gives 
$$(-g(p) - p \cdot 
   g'(p)) \phi'(1 - p \cdot g(p)) 
   = -g'(p).$$  
Rearranging this expression to isolate $g'(p)$, while using 
Proposition \ref{pr:K}, along with the fact that $\phi'(1-x) = \mu-\phi''(1) x
+ o(x)$ as $x\rightarrow 0$, we get
$$g'(p) 
   \; = \; \frac{g(p) \phi'(1-p \cdot g(p))}{1-p\cdot \phi'(1-p \cdot g(p))}
   \; = \; \frac{2\mu^3}{\phi''(1)}+o(1)$$
as $p \downarrow p_c$.
$\Cox$
\end{proof}

Define invasion percolation on an arbitrary tree $T$ as follows.  
Start with $I^T_0 = \bf 0$ where we recall that $\bf 0$ is the root of $T$.  
Inductively define $I^T_{n+1}$ 
to consist of $I^T_n$ along with the vertex of minimal weight $U_v$ 
adjacent to $I^T_n$.  The invasion percolation cluster is defined as 
$I^T := \bigcup I^T_n$.  Note that $I^T$ is measurable with respect
to the $U$ variables.  Let $I := I^\TT$ \label{I-def} denote the invasion cluster
of the random tree $\TT$.  By independence of the $U$ variables and
$\T$, the conditional distribution of $I$ given $\TT$ agrees with that
of invasion percolation.
    
\begin{pr} \label{prop:infst}
For any $p > p_c$, $I$ almost surely reaches some vertex $v$
such that $v \conn_p \infty$ in $\TT (v)$.
\end{pr}

\begin{proof}
We consider the following coupling that generates $I$ at the same time 
as $\TT$: begin with the root, and generate children according to $Z$, 
giving each new edge a $(0,1)$ weight uniformly and independently. 
If there exists a weight below $p$, take the edge of minimal weight 
among all boundary edges and generate children according to $Z$, 
giving each new edge a weight as before.  Continue this process; 
if it terminates, then choose the minimal weight overall (which 
will necessarily be greater than $p$), generate children and 
assign weights as before.  Repeat this process.

Each event measurable with respect to the law described above has the 
same probability as with respect to $\P$.  At the step after 
each termination---when all available weights are greater than 
$p$---the event that the newly generated vertex has an infinite 
subtree is independent from all that came before and has positive 
probability $g(p)$.  By the second Borel-Cantelli lemma, there will necessarily be an invaded edge 
with an infinite subtree below it with weights less than $p$.  $\Cox$
\end{proof}


\begin{cor} \label{cor-finite}
For any $p > p_c$, the number of edges in $I$ with weight greater than $p$ 
is almost surely finite.
\end{cor}

This was proven for a large class of graphs by H\"aggstr\"om, Peres and 
Schonmann~\cite{HPS}, but this class doesn't cover the case of 
Galton-Watson trees conditioned on survival; they exploit quite 
a bit of symmetry that does not occur 
in the Galton-Watson case.
    
\begin{proof}
Let $x$ be the first invaded vertex with an infinite subtree below 
with weights less than $p$.  Then after $x$ is invaded, no edges 
of weight larger than $p$ will be invaded.    $\Cox$
\end{proof}

\begin{cor} \label{cor:path}
There is almost surely only one infinite non-backtracking path from 
$\bf 0$ in $I$.  Equivalently, $\TT$ is almost surely in the set of
trees $T$ such that $I^T$ contains almost surely a unique infinite
non-backtracking path from $\rtt$.
\end{cor}    

\begin{proof} 
Suppose that there are two distinct paths to infinity 
in $I$; by Corollary \ref{cor-finite}, there exist maximal weights 
$M_1$ and $M_2$ along these paths after they split, $\P$-almost surely.
If $M_1 > M_2$, the second infinite path would be invaded before the edge 
containing $M_1$.  Similarly, we cannot have $M_2 > M_1$.  
Finally, $M_1 = M_2$ has  $\P$-probability $0$, completing the proof. $\Cox$
\end{proof}

This proof is stated for invasion percolation on regular trees 
in~\cite{angel-invasion2008}, but is identical for Galton-Watson trees 
once Corollary~\ref{cor-finite} is in place; the unique path guaranteed by 
\ref{cor:path} is typically called the backbone of $I$, and we continue this 
convention.  Note that 
a regular tree is simply a Galton-Watson tree with $Z \equiv b$.

\begin{defn}[the invasion path $\gamma$]
Let \label{gamma-def} $\gamma^T := (\rtt, \gamma^T_1, \gamma^T_2, \ldots)$ be the 
random sequence whose $n$th element is the unique $v$ with $|v|=n$ 
 such that $v \leftrightarrow \infty$ via a downward path 
in the invasion cluster $I^T$.  
Let $\nu_T$\label{nu_T-def} denote the law of $\gamma^T$ given $T$.  Let $\nu_\TT$
denote the random measure on the random space $(\TT , \partial \TT)$
induced by the $\gamma^\TT$.  In other words, for measurable
$A \subseteq \partial \ulam$, $\nu_\TT (A , \omega) = \P [\gamma^T \in A]$
evaluated at $T = \TT(\omega)$.  By Corollary~\ref{cor:path}, this is
a well defined probability measure for almost every $\omega$.
\end{defn}

\subsection{Preliminary comparison of limit-uniform and invasion measures}

Our main goal is to see whether $\nu_\TT$ is almost surely 
absolutely continuous with respect to $\mu_\TT$.  The following easy 
first result shows that $\nu_\TT$ has the same entropy as $\mu_\TT$. 

\begin{defn}  Let $T$ be a rooted tree with branching number $\mu$ and $S \subset \partial T$.  For each $n$, define $S_n$ to be the set of vertices in paths in $S$ at height $n$.  If there exists an $\ee > 0$ so that $|S_n| \lesssim (\mu - \ee)^n$, then we say that $S$ is \emph{exponentially small}.  If every exponentially small set $S$ has $\nu_T(S) = 0$, then we say that the \emph{the limit uniform and invasion measures on $T$ have the same entropy.}
\end{defn}

\begin{thm} \label{th:small}
For $\GW$-a.e. $\TT$, the limit uniform and invasion measures on $\TT$ have the same entropy.
\end{thm} 

\begin{proof}
Condition on $\T$ and fix some exponentially small $S \subset \partial \TT$. 
Let $p \in (\frac{1}{\mu},\frac{1}{\mu-\ee})$; then since 
$p > p_c$, Proposition \ref{prop:infst} implies that there exists 
an invaded vertex with an infinite subtree of weights less than $p$ 
below it.  Let $v_p$ be the first such invaded vertex.  Then the 
backbone is contained in the subtree below $v_p$ since the parental 
edge of $v_p$ together with edges of weight larger than $p$ separate 
the root from infinity.  Moreover, all invaded vertices below $v_p$ 
have parental edge weight at most $p$, implying that the backbone is 
contained in the set of open edges below $v_p$ with weight at most $p$; 
call this subtree $C_p$.  Note that the law of the subtree below $v_p$ 
(unconditioned on $\T$) with weight at most $p$ has the law of a 
Galton-Watson tree with progeny distribution $\Bin(Z,p)$ conditioned 
on $v_p$ being contained 
in an infinite open cluster.  Thus, conditioned on $\T$ and $v_p$, 
the law of this subtree is the law of the open cluster containing 
$v_p$ of Bernoulli($p$) percolation on the subtree below $v_p$ 
conditioned on this cluster being infinite.  Note that since $\TT$ 
is Galton-Watson, the critical parameter for percolation on the 
subtree below $v_p$ is $p_c$; in particular, if we call this cluster 
$C_p$, we have that $\P_*[|C_p| = \infty \| v_p] > 0$ $\nu_\TT$-a.s.      

\noindent Condition on $\T$ and condition further on $v_p$; note that 
for $x$ below $v_p$ we have 
\begin{equation}\label{eq:union-bd}
\P_*[x \in C_p \| v_p, |C_p| = \infty] 
   \leq \frac{p^{|x| - |v_p|}}{\P_*[|C_p| = \infty \| v_p]}.
\end{equation}
Thus for any $n$, we have 
$$\P_*[S \cap \partial I \neq \emptyset \| v_p] 
   \leq \P_*[S_n \leftrightarrow v_p \| v_p,|C_p| = \infty] 
   \leq \frac{|S_n| p^{n-|v_p|}}{\P_*[|C_p| = \infty \| v_p] }$$
where the last inequality is via a union bound over $S_n$.  
By the definition of $p$, $|S_n|p^n \to 0$; therefore, for each $v_p$, 
we have 
$$\P_*[S \cap \partial I \neq \emptyset \| v_p] = 0.$$
Since there are $\P$-a.s. only countably many vertices in $\TT$, 
the above gives 
\begin{equation}\label{eq:mu-eq-zero}
\nu_\TT (S) = \P_*[S \cap \partial I \neq \emptyset ] 
   \leq \sum\limits_{v \in \TT } 
   \P_*[S \cap \partial I \neq \emptyset \| v_p = v] = 0
\end{equation} 
$\P$-a.s. thereby completing the proof. $\Cox$
\end{proof}

In the remainder of this section, we give the summability criterion
that establishes a sufficient condition for absolute continuity
in terms of the KL-divergence of the two measures along a ray chosen
from $\nu_\TT$.  

\begin{defn}[the splits $p$ and $q$ at children of $u$, 
and their difference, $X$] \label{def:splits}
Let $v$ be a vertex of $\TT$ and let $u$ be the parent of $v$.  Define 
\begin{eqnarray*}
p(v) & := & \mu_\TT (v)/\mu_\TT (u) \label{p(v)-def}\\
q(v) & := & \nu_\TT (v)/\nu_\TT (u) \label{q(v)-def}\\
X(u) & := & \sum_w q(w) \log [q(w) / p(w)] \label{X-def} 
\end{eqnarray*}
\end{defn}
where the sum is over all children $w$ of $u$ and $\nu_\TT(v) = \nu_\TT(\{\gamma : v \in \gamma \})$ and $\mu_\TT(v)$ is defined similarly.  The quantity $X$
is known as KL-divergence.  The KL-divergence $\KL (\rho,\rho')$
is defined between any two probability measures $\rho$ and $\rho'$
on a finite set $\{ 1 , \ldots , k \}$ by the formula
$$\KL (\rho , \rho') := \sum_{i=1}^k \rho'(i) \log \frac{\rho'(i)}{\rho(i)} \, .$$
It is a measure of the difference between the two distributions.
It is always non-negative but not symmetric.  The following inequality
shows that $\KL$ behaves like quadratic distance away from $\rho = 0$.

\begin{pr} \label{pr:KL}
Let $\rho$ and $\rho'$ be probability measures on the set 
$\{ 1, \ldots , k \}$ and denote  
$\ee_i := \rho'(i) / \rho(i) - 1$.  Then 
\begin{equation}
\KL (\rho , \rho') \leq \sum_{i=1}^k \rho(i) \ee_i^2 \, .
\end{equation}
\end{pr}

\begin{proof}
Define the function $R$ on $(-1,\infty)$ by 
$$R(x) := \frac{(1+x) \log (1+x) - x}{x^2}$$
if $x \neq 0$ and $R(0) := 1/2$.  This makes $R$ continuous, 
positive, and decreasing from~1 to~0 on $(-1,\infty)$.  When 
$\ee = \rho'/\rho - 1$, we may compute
$$\frac{\rho' \log (\rho'/\rho) - (\rho'-\rho)}{\rho} 
   = \frac{(1+\ee) \rho \log (1+\ee) - \ee \rho}{\rho} = \ee^2 R(\ee) \, .$$
Because $\sum_{i=1}^k \rho(i) = \sum_{i=1}^k \rho'(i) = 1$, we see that
$$\KL (\rho, \rho') 
   = \sum_{i=1}^k (\rho'(i) - \rho(i)) + \rho(i) \ee_i^2 R(\ee_i) 
   = \sum_{i=1}^k \rho(i) \ee_i^2 R(\ee)$$
and the result follows from $0 < R(\ee) < 1$.
$\Cox$
\end{proof}

Applying Proposition \ref{pr:KL} to $\rho' = q$ and $\rho = p$ gives \begin{equation} \label{eq:KL}
    X(u) \leq \sum_{w}p(w)\ee(w)^2
\end{equation}
where $\ee(w) = \frac{q(w)}{p(w)} - 1.$  

\begin{lem} \label{lem:X}
Let $T$ be a fixed tree on which $\nu_T$ and $\mu_T$ are well defined 
on the Borel $\sigma$-field $\B$ on $\partial T$.  If
\begin{equation} \label{eq:L1}
\sum_{n=1}^\infty \sum_{|v|=n} X(v) \nu_T (v) < \infty
\end{equation}
then $\nu_T \ll \mu_T$.  
\end{lem}

\begin{proof}
On the measure space $(\partial T , \B),$ define a filtration 
$\{ \G_n \}$ by letting $\G_n$\label{G_n-def} denote the $\sigma$-field generated 
by the sets $\{ \gamma : \gamma_n = v \}$.  The Borel $\sigma$-field 
$\B$ is the increasing limit $\sigma ( \bigcup_n \G_n )$.  Let
$$M_n := \left. \frac{d\nu}{d\mu} \right |_{\G_n} \, .\label{M_n-def}$$
In other words, $M_n (\gamma) = \nu (\gamma_n) / \mu (\gamma_n)$. 
Let $\Mbar := \limsup_{n \to \infty} M_n$.  The Radon-Nikodym 
martingale theorem~\cite[Theorem~5.3.3]{durrett4} says that 
$\{ M_n \}$ is a martingale with respect to $(\partial T,
\B, \mu_T , \{ \G_n \})$ and that $\nu_T \ll \mu_T$ is equivalent to 
$\nu_T( \{ \Mbar = \infty \} )= 0$.  This is equivalent to 
$\nu_T (\{ \ul{M}  = 0\}) = 0$ where $\ul{M} = 1 / \Mbar = \liminf_n 1/M_n$.
The sequence $\{ 1/M_n \}$ is a $\nu_T$-martingale, therefore
$\{ \log (1/M_n) \}$ is a $\nu_T$-supermartingale and to conclude
that it $\nu_T$-a.s.\ does not go to negative infinity, it suffices 
to show that its expectation is bounded from below.  

We compute the conditional expected increment of $\log (1/M_n)$. 
Letting $\gamma$ denote the ray $(\gamma_1 , \gamma_2, \ldots)$, 
$$\log \frac{1}{M_{n+1} (\gamma)}  - \log \frac{1}{M_n (\gamma)} = 
   \log \frac{\nu_T (\gamma_n)}{\mu_T (\gamma_n)} 
   - \log \frac{\nu_T (\gamma_{n+1})}{\mu_T (\gamma_{n+1})} 
   = - \log \frac{q(\gamma_{n+1})}{p(\gamma_{n+1})} \, .$$
Conditioning on $\G_n$, if $\gamma_{n+1} = u$, then the $\nu_T$-probability
of $\gamma_n = v$ is $q(v)$, whence
$$\E_{\nu_T} \left [ \log \frac{1}{M_{n+1}} - \log \frac{1}{M_n} 
   \,\bigg|\, \G_n \right ]
   = \sum_{v \text{ child of } u} - q(v) \log \frac{q(v)}{p(v)} = - X(u) 
   \, .$$
Taking the unconditional expectation,
$$\E_{\nu_T} \left [ \log \frac{1}{M_{n+1}} - \log \frac{1}{M_n} \right ]
   = - \sum_{|v|=n} \nu_T (v) X(v)$$ 
and summing over $n$ shows that~\eqref{eq:L1} implies that $\log (1 / M_n)$
has expectation bounded from below, establishing $\nu_T \ll \mu_T$.  $\Cox$
\end{proof}

\begin{cor} \label{cor:X}
Recall that $\gamma$ denotes the invasion path on $\TT$ and let 
$X_n$ denote $X (\gamma_n)$. 
\begin{enumerate}[(i)]
\item If $\sum_{n=1}^\infty \E X_n < \infty$ then $\nu_\TT \ll \mu_\TT$
$\GW$-almost surely.
\item Define the filtration $\{ \G_n' \}$ on $(\Omega , \F)$ by letting
$\G_n'$ be the $\sigma$-field generated by $\T$ together with $\gamma_1, 
\ldots , \gamma_n$.  Let $Y(v)$ be random variables such that $Y(v) \in \G_{|v|}'$ and
$$\P [X(\gamma_n) \neq Y(\gamma_n) \mbox{ infinitely often} ] = 0 \, .$$
Then $\sum_{n=1}^\infty \E Y_n < \infty$ implies that $\GW$-almost surely,
$\nu_\TT \ll \mu_\TT$.
\end{enumerate}
\end{cor}

\begin{proof} ~~\\[1ex]
$(i)$ Writing $\E X_n = \E \left[ \E_* X_n \right]$ we see that the hypothesis of~$(i)$,
namely $\sum_{n=1}^\infty \E X_n < \infty$, implies $\E \sum_{n=1}^\infty
\E_* X_n < \infty$.  This implies $\sum_{n=1}^\infty \E_* X_n < \infty$
almost surely.  A version of $\E_* X_n$ is $\sum_{|v| \in \TT_n} 
X(v) \nu_\TT (v)$, whence~\eqref{eq:L1} holds for $\GW$-almost 
every $\TT$, implying almost sure absolute continuity of $\mu_\TT$ 
with respect to $\nu_\TT$. 
\\[1ex]
$(ii)$ The argument used to prove Lemma~\ref{lem:X} 
may be adapted as follows.  Let $\disp M_n := \left. 
\frac{d \nu_\TT}{d\mu_\TT} \right |_{\G_n'}$,
a version of which is the function taking the value 
$\disp \frac{\nu_\TT (v)}{\mu_\TT (v)}$ on 
$\{ \gamma_n = v \}$.  Again $\{ M_n \}$ is a martingale and 
$\log (1/M_n)$ is a supermartingale which we need to show converges
almost surely.  The sequence
$$S_M := \sum_{n=1}^M
   \left ( \log \frac{1}{M_{n+1}} - \log \frac{1}{M_n} \right )
   \, \one_{X(\gamma_n) = Y(\gamma_n)}$$
is a convergent supermartingale because its expected increments are 
either~0 or $-Y(\gamma_n)$; convergence of the unconditional expectations
$\E Y (\gamma_n)$ implies almost sure convergence of the expected
increments, implying almost sure convergence of the supermartingale 
$\{ S_M \}$.  The hypotheses of~$(ii)$ imply that the increments of
$S_M$ differ from the increments of $\log (1 / M_n)$ finitely often
almost surely, implying convergence of the supermartingale $\log (1 / M_n)$
and hence $\nu_\TT \ll \mu_\TT$.
$\Cox$
\end{proof}

\section{Survival function conditioned on the tree} \label{sec:g}

This section is concerned with estimating $g(\TT , p)$, a random 
variable measurable with respect to $\T$.  We first prove an upper bound on $g$ which gives a uniform bound on the $L^q$ norm of $g$.  Additionally, we show that conditioning on only the first $n$ levels gives a random variable exponentially close to $g$.  Estimating this averaged random variable is a key element in the proof of Theorem~\ref{th:main}, and is the content of section~\ref{sec:bounds-E}.

The following result from~\cite{LP-book} will be useful.
\begin{thm}[\protect{\cite[Theorem~5.24]{LP-book}}]
\label{th:res-prob}
For independent percolation, we have 
$$\frac{1}{\mathscr{R}(\rtt \leftrightarrow \infty) + 1} 
   \leq \P_*[\rtt \leftrightarrow \infty] 
   \leq \frac{2}{\mathscr{R}(\rtt \leftrightarrow \infty) + 1} $$
where $\mathscr{R}(\bf {0} \leftrightarrow \infty)$ denotes the 
effective resistance from $\bf 0$ to infinity when each edge $e$ 
with percolation parameter $p_e$ is given resistance 
$$r(e) = \frac{1 - p_e}{\prod\limits_{o < x \leq u}p_x}.$$
$\Cox$
\end{thm}

From this, we deduce:
\begin{pr} \label{pr:g-upper-bound} For any $\ee \in (0,1 - p_c)$ and $\GW$-almost surely,
\begin{equation}
    g(\TT , p_c + \ee) < \frac{2 \ee \overline{W}}{(1 - p_c - \ee)p_c} 
\end{equation}
where $\overline{W} := \sup\limits_n W_n(\TT)$ \label{Wbar-def} is almost surely 
finite because $\lim_{n \to \infty} W_n$ exists almost surely. 
\end{pr}

\begin{proof}
To get an upper bound on $g$, we need a lower bound on the resistance.  For each height $n$, short together all nodes at this height.  For every $p = p_c + \ee$ this gives a lower bound of \begin{align*}
    \mathscr{R}(\rtt \leftrightarrow \infty) &\geq \sum\limits_{n = 1}^\infty \frac{1 - p_c - \ee}{Z_n (p_c + \ee)^n} \\
    &= (1 - p_c - \ee) \sum\limits_{n = 1}^\infty \frac{p_c^n}{W_n(p_c + \ee)^n} \\
    & \geq \frac{(1 - p_c - \ee)}{\overline{W}} \sum\limits_{n = 1}^\infty \frac{p_c^n}{(p_c + \ee)^n} \\
    &=  \frac{(1 - p_c - \ee)p_c}{\overline{W}\ee}. 
\end{align*}
Using Theorem \ref{th:res-prob}, we get \begin{equation}
    g(\TT,p_c + \ee) \leq \frac{2}{1 + \frac{(1 - p_c - \ee)p_c}{\overline{W} \ee} } \leq \frac{2\ee \overline{W}}{(1 - p_c - \ee)p_c}.
\end{equation}
\end{proof}

\begin{pr} [uniform $L^q$ bound] \label{pr:unif L^q}
Suppose the offspring has a finite $q$ moment.  Then for any $\delta > 0$, there is 
a constant $c_q$ such that for all $\ee \in (0, 1 - p_c - \delta)$,
$$\E g(\TT, p_c + \ee)^q \leq c_q \ee^q \, $$ 
where $c_q = c_q(\delta) > 0.$
\end{pr} 

\begin{proof}
First recall that if the offspring has a finite $q$-moment, 
then $M_q := \E W^q$ is finite as well.  By the $L^q$ maximal 
inequality (e.g.,~\cite[Theorem~5.4.3]{durrett4}), we have that 
$$\E\left[\left(\sup\limits_{1 \leq k \leq n} W_k \right)^q\right] 
   \leq \left ( \frac{q}{q-1} \right )^q \E W_n^q 
   \leq \left ( \frac{q}{q-1} \right )^q M_q$$
because $\{ W_n^q \}$ is a submartingale.

Note that $\disp \left(\sup_{1 \leq k \leq n} W_k\right)^q 
\uparrow \overline{W}^q$ as $n \to \infty$.  By monotone convergence, 
this implies $\E[ \overline{W}^q] \leq (q/(q-1))^q M_q$.  
In particular, for any $\ee < 1 - p_c - \delta$, this implies 
$$\E[g(\TT,p_c + \ee)^q] \leq \left(\frac{2\ee}{(1 - p_c - \delta)p_c}
   \right)^q \E[\overline{W}^q] 
\leq \left(\frac{2\ee}{(1 - p_c - \delta)p_c}\right)^q 
      \left ( \frac{q}{q-1} \right )^q M_q \, ,$$
proving the proposition with $\disp c_q = \left ( 
\frac{2q}{(q-1)(1-p_c - \delta) p_c} \right )^q M_q$.
$\Cox$
\end{proof}

Let $\T_n$ denote $\sigma (\deg_v : |v| \leq n)$.  Because $\T_n 
\uparrow \T$ and $g$ is bounded, we know that $\E [ g(\TT,p) \| \T_n ]
\to g(\TT , p)$ almost surely and in $L^1$.  It turns out that
$g_n := \E [g(\TT,p) \| \T_n]$ is much easier to estimate that $g$ itself.
Our strategy is to show this convergence is exponentially rapid, 
transferring the work from estimation of $g$ to estimation of $g_n$.

\begin{lem} \label{lem:mean}
For any $\delta > 0$, there are constants $c_i > 0$ such that for all $p \in (p_c,\sqrt{p_c} - \delta)$ 
$$\big | \, g(\TT,p) - g_n(\TT,p) \, \big | \leq c_1 
   e^{-c_2 n} $$
with probability at least $1 - e^{-c_3 n}$.
\end{lem} 

\begin{proof}  Define a random set $S = S(n,p)$
to be the set of vertices $v \in \TT_n$ such that $\rt \conn_p v$.
Let $\pi_\TT$ denote the law of the random variable $S$, an atomic 
probability measure on the subsets of the random set $\TT_n$.
Using
$$g(\TT,p) = \P[H(p) \| \T] = 
   \E \left [ \P [H(p) \| \F_n']  \| \T \right ] $$
where $\F_n'$ \label{F_n'-def} be the $\sigma$-field generated by $\F_n$ and $\T$,  
we obtain the explicit representation
\begin{equation} \label{eq:f}
g(\TT,p) = \sum_S \pi_\TT (S) \left [ 1 - \left ( \prod_{v \in S}
   (1 - g(\TT(v) , p) ) \right ) \right ] \, .
\end{equation}

Order the vertices in $\TT_n$ arbitrarily and define the revealed 
martingale $\{ M_k \}$ by
\begin{equation}
M_k := \E \left [ g(\TT,\ee) \| \T_n , \{ \TT(v_j) : j \leq k \} \right ]
\end{equation}
as $k$ ranges from~0 to $|\TT_n|$.  By definition, $M_0 = g_n$.  
Also, $M_{|\TT_n|} = g(\TT,p)$ because from $\T_n$ together with
$\{ \TT(v) : v \in \TT_n \}$ one can reconstruct $\TT$.  Arguing as in 
\eqref{eq:f}, we obtain the explicit representation 
\begin{equation} \label{eq:M_k-rep}
M_k = \sum_S \pi_{\TT}(S) \left[1 - \prod_{v \in S_{\leq k}}(1 - g(\TT(v),p)
(1 - g(p))^{|S_{>k}|} \right]
\end{equation}
where for a given set $S \subset \TT_n$,  $S_{\leq k }$ denotes the vertices in $S$
indexed $\leq k$ and $S_{> k}$ denotes the set indexed $> k$.  

We claim the increments of $\{ M_k \}$ are bounded by $p^n$.  
Indeed, \eqref{eq:M_k-rep} implies 
$$|M_{k+1} - M_k| \leq \sum_{S \ni v_k} \pi_{\TT(S)}|g(\TT(v_k),p) - g(p)| 
\leq \sum_{S \ni v_k} \pi_{\TT(S)} = \P[\rtt \conn_p v_k] = p^n.$$

Now the lemma without the factor of $\ee$ on the right-hand side
follows from Azuma's inequality \cite{azuma}: the bounded differences yield
Gaussian tails on $|M_{|\TT_n|} - M_0|$ with variance 
$|\TT_n| p^{2n}$, which is exponentially small in $n$ when $\mu p^2 < 1$. $\Cox$
\end{proof}

\section{Pivot Sequence on the Backbone} \label{sec:backbone}

\subsection{Markov property}

Define the shift function $\theta : \Omega \to \Omega$ by 
\begin{equation} \label{eq:theta}
(\theta (\omega))_v := \omega_{\gamma_1 \sqcup v} \, .
\end{equation}
Informally, $\theta$ shifts the values of random variables at
nodes $\gamma_1 \sqcup v$ in $T(\gamma_1)$ back to node $v$; 
these values populate the whole Ulam tree; values of variables 
not in $T(\gamma_1)$ are discarded; this is a tree-indexed version 
of the shift for an ordinary Markov chain.  The $n$-fold
shift $\theta^n$ shifts $n$ steps down the backbone.  This
subsection is devoted to several versions of the Markov property
involving the shift $\theta$ as well as the limiting behavior of $\pivot_n:=\beta(\gamma_n)$. \label{pivot_n-def}

\begin{defn}[dual trees and pivots] \label{def:dual-piv}
Recall that $T(v)$ denotes the subtree from $v$, moved to the root.
Let $T^*(v)$ \label{dual-tree-def} denote the rooted subtree induced on all vertices 
$w \notin T(v)$, and let $\dual_{v,w}$ represent the
pivot of the vertex $w$ on $T^*(v)$, that is, the least $x$ such that $w$ is 
connected to infinity without going through $v$.  The dual pivot \label{dual-def} $\dual_v$ is defined to be $\underset{w<v}{\text{min}}\ \dual_{v,w}$.  In keeping
with the notation for pivots, we denote $\dual_n := \dual_{\gamma_n}$. \label{dual-n-def}
\end{defn}

\begin{defn} We define the following $\sigma$-fields. \\[-5ex]
\begin{enumerate}[(i)]
\item For fixed $v \neq \rtt$, define \label{CC_v-def} $\CC_v$ to be the $\sigma$-field 
generated by $\deg_w$ and $U_w$ for all $w \neq v$ in $T(v)$ along with 
$\deg_v$.  Define \label{BB_v^*-def} $\BB_v^*$ to be the $\sigma$-field generated by
all the other data: $U_w$ and $\deg_w$ for all $w \in \T^* (v)$, along
with $U_v$.
\item For $n \geq 1$, let \label{BB_n^*-def} $\BB_n^*$ denote the $\sigma$-field containing
$\gamma_n$ and all sets of the form $\{ \gamma_n = v \} \cap B$ where
$B \in \BB_v^*$. Informally, $\BB_n^*$ is generated by $\gamma_n$
and $\BB_{\gamma_n}^*$.  
\item Let $\CC_n$ \label{CC_n-def} be the $\sigma$-field generated by $\theta^n \omega$;
in other words it contains $\deg (\gamma_n)$ and all pairs 
$(\deg_{\gamma_n \sqcup x} , U_{\gamma_n \sqcup x})$.  It is not 
important, but this definition does not allow $\CC_n$ to know the 
identity of $\gamma_n$.  
\end{enumerate}
\end{defn}

It is elementary that $\{ \BB_n^* \}$ is a filtration, that
$\BB_n^* \cap \CC_n$ is trivial, and that $\BB_n^* \vee \CC_n = \F$.  

\begin{defn} We define the following conditioned measures. \\[-5ex]
\begin{enumerate}[(i)]
\item For $x \in (p_c,1)$, let $\QQ_x := (\P \| \pivot_0 \leq x)$ denote the conditional law given
$\rtt \leftrightarrow_x \infty$, in other words, $\QQ_x [A] = 
\frac{g_A (x)}{g(x)}$ where
$$g_A (x) := \P [A \cap \{ \pivot_0 \leq x \} ] \, .$$
\item Let $\cL$ denote the law of $\pivot_0$, the pivot at the root.
By~\cite[Theorem~5.1.9]{durrett4}, one may define regular 
conditional distributions $\P_x := (\P \| \pivot_0 = x)$.  These 
satisfy $\P_x [\pivot_0 = x] = 1$ and $\int \P_x \, \,d\cL(x) = \P$.
Also, $\QQ_y = (1/g(y)) \int \P_x \,d\cL|_{[0,y]} (x)$.
\end{enumerate}
\end{defn}

A common null set for all the conditioned measures is the set where 
either the invasion ray is not well defined or $\pivot (v) = \dual_v$
for some $v$.  Statements such as~\eqref{eq:pivot < dual} below are
always interpreted as holding modulo this null set.

\begin{pr}[Markov property for dual pivots] \label{pr:markov 2}
~~\\[-3ex]
\begin{enumerate}[(i)]
\item For any $A \in \F$,
$$\P [\theta^n \omega \in A \| \BB_n^*] = \QQ_{\dual_n} [A] \, .$$
\item More generally, if $0 < y \leq 1$ then for any $A \in \F$,
$$\QQ_y [\theta^n \omega \in A \| \BB_n^*] = \QQ_{\dual_n \wedge y} [A] \, .$$
\item Under $\P$, the sequence $\{ \dual_n \}$ is a time homogeneous Markov
chain adapted to $\BB_n^*$ with transition kernel $p(x,S) = 
\QQ_x [\beta_1^* \wedge x \in S]$ and initial distribution $\delta_1$.  
\end{enumerate}
\end{pr}

\noindent{\sc Proof:}
$(i)$ By definition of conditional probability, the conclusion is
equivalent to $\P [\theta^n \omega \in A ; G] = \int_G \QQ_{\beta_n^*} [A] 
\, d\P$ for all $G \in \BB_n^*$.  Writing $G$ as the countable union of
$\BB_n^*$-measurable sets $\bigcup_v (G \cap \{ \gamma_n = v \})$,
it suffices to verify the previous identity for each 
piece $G \cap \{ \gamma_n = v \}$.  By the definition of $\BB_n^*$,
each of these may be written as $\{ \gamma_n = v \} \cap B^*$ for
$B^* \in \BB_v^*$.  Thus, it suffices to prove
\begin{equation} \label{eq:cond prob 2}
\int_{\{ \gamma_n = v \} \cap B^*} \QQ_{\dual_v} [A] \, d\P 
   = \int_{\{ \gamma_n = v \} \cap B^*} \one_A (\theta^n \omega) \, d\P
\end{equation}
for all $v \in T_n$, $B^* \in \BB_v^*$ and $A \in \F$.  

Fixing $v$, identify $(\Omega , \F , \P)$ as a product space
$(\Omega_1 , \F_1 , \P_1) \times (\Omega_2 , \F_2 , \P_2)$ 
where $\Omega_1 = (\N \times [0,1])^{T(v) \setminus \{ v \}} \times
\N$ and $\Omega_2 = (\N \times [0,1])^{T^*(v)} \times [0,1]$.  Let
$\pi_i : \Omega \to \Omega_i$ denote the coordinate maps; then
$\pi_2$ is a measure preserving map on $(\Omega , \BB_v^* , \P)$
and $\pi_1$ is a measure preserving map on $(\Omega , \CC_v , \P)$.
In particular, $B^* = \pi_2^{-1} B$ for some $B \in \F_2$.

Working on the left-hand side of~\eqref{eq:cond prob 2}, observe
that $\pivot \circ \pi_2 = \dual_v$ and hence, if we let $H$ represent the event that the invasion percolation ever gets to $v$, then we have
\begin{equation} \label{eq:pivot < dual}
\{ \gamma_n = v \} = \{ \pivot (v) \leq \dual_v \}\cap H
   = \{ \pivot_0 (\pi_1 \omega) < \dual_v (\pi_2 \omega) \}\cap H\, .
\end{equation}
Using this, we obtain
\begin{eqnarray*}
\int_{\{ \gamma_n = v \} \cap B^*} \QQ_{\dual_v} [A] \, d\P 
& = & \int_\Omega \one_B (\pi_2 \omega)
   \one_{\left\{\pivot (\pi_1 \omega) < \dual_v (\pi_2 \omega)\right\}\cap H}
   \QQ_{\dual_v (\pi_2 \omega)} [A] \, d\P \\[1ex]
& = & \int_{\Omega_2} \one_B (\omega_2) \QQ_{\dual_v (\omega_2)} [A]
   \left [ \int_{\Omega_1} \one_{\left\{\pivot (\omega_1) < \dual_v (\omega_2)\right\}\cap H_2}
   \, d\P_1 (\omega_1) \right ] \, d\P_2 (\omega_2) \\[1ex]
& = & \int_{\Omega_2} \one_B (\omega_2) 
   \frac{g_A (\dual_v (\omega_2))}{g(\dual_v (\omega_2))} 
   \left [ \int_{\Omega_1} \one_{\left\{\pivot (\omega_1) < \dual_v (\omega_2)\right\}\cap H_2}
   \, d\P_1 (\omega_1) \right ] \, d\P_2 (\omega_2) 
\end{eqnarray*}
where $H_2$ above denotes $\pi_2(H)$.  The integral over $\Omega_1$ is equal to $g(\dual_v (\omega_2)) \one_{H_2}(\omega_2)$ 
so we may simplify and continue.
Writing $g_A (x)$ as $\int_{\Omega_1} \one_{A \cap \{ \pivot_0 < x \}}
(\omega) \, d\P_1 (\omega)$ in the second line,~\eqref{eq:cond prob 2}
is finished as follows.
\begin{eqnarray*}
\phantom{\int_G \QQ_{\dual_v} [A] \, d\P} & = & 
   \int_{\Omega_2} \one_B (\omega_2) g_A (\dual_v (\omega_2)) \one_{H_2}(\omega_2) \, d\P_2  \\[1ex]
& = & \int_{\Omega_1 \times \Omega_2} \one_B (\omega_2) \one_A (\omega_1)
   \one_{\pivot (\omega_1) < \dual_v (\omega_2)} \one_{H_2}(\omega_2) \, d\P_1 (\omega_1)
   \, d\P_2 (\omega_2) \\[1ex]
& = & \int_{\{ \gamma_n = v \} \cap B^*} \one_A (\pi_1 \omega) \, d\P \\[1ex]
& = & \int_{\{ \gamma_n = v \} \cap B^*} \one_A (\theta^n \omega) \, d\P
\end{eqnarray*}
because $\pi_1 \omega = \theta^n \omega$ on $\{ \gamma_n = v \}$.

$(ii)$ Begin with the observation that 
\begin{equation} \label{eq:decomp}
\{ \pivot_0 < y \} \cap \{ \gamma_n = v \} = 
   \{ \pivot (\pi_1 \omega < y) \} \cap 
   \{ U_w < y \mbox{ for all } w \leq v \} \cap 
   \{ \pivot (\pi_1 \omega) < \dual_v (\pi_2 \omega) \}\cap H \, .
\end{equation}
As before, letting $G = \{ \gamma_n = v \} \cap \pi_2^{-1} (B)$,
we aim to prove the second identity in
\begin{equation} \label{eq:cond prob 3}
\int_G \QQ_{\dual_v \wedge y} [A] \, d\QQ_y = 
\int_G \frac{g_A (\dual_v \wedge y)}{g (\dual_v \wedge y)} \, d\QQ_y
   = \int_G \one_A (\theta^n \omega) \, d\QQ_y \, ,
\end{equation}
the first being definitional.  Also by definition, $\QQ_y [\cdot] = 
(1/g(y)) \P [\cdot \cap \{ \pivot_0 < y \}]$, whence, 
using~\eqref{eq:decomp}, the left-hand side of~\eqref{eq:cond prob 3} 
becomes
$$\frac{1}{g(y)} \int_{\Omega_1 \times \Omega_2} 
   \one_{B\cap H_2} (\omega_2) \; \one_{U_w < y \forall w \leq v} (\omega_2) \; 
   \frac{g_A (\dual_v (\omega_2) \wedge y)}{g (\dual_v (\omega_2) \wedge y)} 
   \; \one_{\pivot (\omega_1) < y} 
   \; \one_{\pivot (\omega_1) < \dual_v (\omega_2)} \; d(\P_1 \times \P_2) \, .$$ 
Integrating over $\Omega_1$ turns the last two indicator functions into
$g(\dual _v (\omega_2) \wedge y)$, again canceling the denominator and
yielding
$$\frac{1}{g(y)} \int_{\Omega_2}
   \one_B (\omega_2) \; \one_{U_w < y \forall w \leq v} (\omega_2) \; 
   g_A (\dual_v (\omega_2) \wedge y) \one_{H_2}(\omega_2) \, d\P_2 (\omega_2) \, .$$
Rewriting $g_A (x)$ as $\int_{\Omega_1} \one_{A \cap \{ \pivot_0 < x \}} 
\, d\P_1$, this becomes
$$\frac{1}{g(y)} \int_{\Omega} 
   \one_B (\pi_2 \omega) \; \one_{U_w < y \forall w \leq v} (\pi_2 \omega) \; 
   \one_{\pivot (\pi_1 \omega) < \dual_v (\pi_2 \omega)} \;
   \one_{\pivot (\pi_1 \omega) < y} \;
   \one_A (\omega_1) \one_{H_2}(\omega_2) \; d\P \, .$$
Observing that the first three indicator functions define $G$, this
simplifies to
\begin{eqnarray*} 
\frac{1}{g(y)} \int_{G \cap \left\{ \pivot_0 (\omega) < y \right\}}
   \one_A (\pi_1 \omega) \, d\P
& = & \frac{1}{g(y)}\int_G \one_A(\pi_1 \omega)
   \one_{\pivot_0 < y}(\omega) \, d\P \\[1ex]
& = & \int_G \one_A(\pi_1 \omega) \, d\QQ_y \\[1ex]
& = & \int_G \one_A(\theta^n \omega) \, d\QQ_y \, 
\end{eqnarray*}
where the last inequality follows from the fact that
$\gamma_n=v$ on $G$, which implies $\pi_1 \omega = \theta^n \omega$.
Hence, that completes the proof of $(ii)$.

$(iii)$ Begin by observing that $\gamma_{n+1} = v \sqcup j$ 
if and only if $\gamma_n = v$ and $\pivot (v \sqcup j) < 
\dual_1 \circ \pi_1$.  In other words, given that the backbone 
contains $v$, the next backbone vertex depends only on $\theta^n \omega$
and is chosen in the same way the first backbone vertex of $\omega$
was chosen.  From this it follows that
$$\dual_{n+1} = \dual_n \wedge \dual_1 \circ \theta^n \, .$$
Therefore,
\begin{eqnarray*}
\P [\dual_{n+1} \in S \| \BB_n^*] & = & 
   \P [\dual_n \wedge (\dual_1 \circ \theta^n) \in S \| \BB_n^*] \\[1ex]
& = & \P [\theta^n \omega \in \{ \omega' : \dual_1 \omega' \wedge 
   \dual_n (\omega) \in S\} \| \BB_n^*] \\[1ex]
& = & \QQ_y [\dual_1 \wedge y \in S] \|_{y = \dual_n}
\end{eqnarray*}
as desired.
$\Cox$

\medskip
\noindent $Remark.$ Note that the final equality in the proof of $(iii)$ does not immediately
follow from the proof of $(i)$ since $\{\omega':\dual_1\omega' \wedge \dual_n\omega\in S \}$ is not a fixed set, but rather depends on
$\dual_n$.  Nevertheless, the proof of this equality follows from a slight modification of the proof of $(i)$, where we simply replace
the expressions $\QQ_{\dual_v(\pi_2\omega)}[A]$, $g_A(\dual_v(\omega_2))$, and $\one_A(\pi_1\omega)$ by the expressions
$\QQ_{\dual_v(\pi_2\omega)}[\dual_1 \wedge \dual_v \in S]$, $\P[\{\pivot_0<\dual_v(\omega_2)\}\cap \{\dual_1 \wedge \dual_v\in S \}]$, and 
$\one_{\{\dual_1(\pi_1\omega) \wedge\dual_v(\omega)\in S \} }$ respectively.

It is immediate that $\QQ_x \ll \P$ for all $x$.  The following
more quantitative statement will be useful.

\begin{pr} \label{pr:variance}
Let $q > 1$ and suppose that the offspring distribution has a finite
$q$-moment.  Then there exists a constant $C_q$ such that for all
$A \in \T$ and for all $\delta > 0$ and all $x \in (p_c,1)$,
$$\P [A] \leq \delta \; \text{ implies } \; \QQ_x [A] 
   \leq C_q \delta^{1 - 1/q} \, .$$
\end{pr}

\begin{proof}
On $\T$, the density of $\QQ_x$ with respect to $\P$ is given by
$$\frac{d\QQ_x}{d\P} (T) = \frac{g(T,x)}{g(x)} \, .$$
Combining Corollary~\ref{cor:dersurvprob}, which implies $g(x) \sim K x$,
with Proposition~\ref{pr:unif L^q}, which shows $\int g(T,p_c + \ee)^q 
\, d\GW(T) \leq c_q \ee^q$ provided $p_c + \epsilon$ is bounded away from $1$, we see that 
$$\int \left | \frac{d\QQ_x}{d\P} (T) \right |^q \, d\GW (T) \leq c_q'$$
for some constant $c_q'$ and all $x \in (0,1)$.  Applying H{\"o}lder's 
inequality with $1/p + 1/q = 1$ then gives 
$$\QQ_x [A] = \int \one_A \frac{d\QQ_x}{d\P} \; d\P
   \leq \left [ \int \one_A \, d\P \right ]^{1/p} 
   \left [ \int \left ( \frac{d\QQ_x}{d\P} \right )^q \, d\P \right ]^{1/q}
   \leq C_q \delta^{1-1/q} $$
when $C_q = (c_q')^{1/q}$.
$\Cox$
\end{proof}

The measures $\P_x$ are in some sense more difficult
to compute with than $\QQ_x$ because of the conditioning 
on measure zero sets.  Relations such as the Markov property,
however, are conceptually somewhat simpler.  The following 
statement of the Markov property generalizes what was proved 
in~\cite[Theorem~1.2~and~Proposition~3.1]{angel-invasion2008}, 
with \label{BB_n^+-def} $\BB_n^+$ representing the $\sigma$-field generated by 
$\BB_n^*$ together with $\pivot_n$.  

\begin{pr}[Markov property for pivots] \label{pr:markov}
For any $A \in \F$, 
$$\P \left [ \theta^n \omega \in A \| \BB_n^+ \right ] 
   = \P_{\pivot_n} [A]$$
on $(\Omega , \F , \P_x)$.
\end{pr}


\begin{proof}
By definition of conditional probability, the conclusion is 
equivalent to
\begin{equation} \label{eq:cond prob}
\P [B \cap \{ \theta^n \omega \in A \}] = \int_B \P_{\pivot_n} [A] \, d\P \, 
\end{equation}
holding for all $B \in \BB_n^+$ and $A \in \F$.  It is enough to 
prove~\eqref{eq:cond prob} for sets that are subsets of 
$\{ \gamma_n = v \}$ for some $v$: if it holds for sets of this form, 
then
\begin{eqnarray*}
\P [B \cap \{ \theta^n \omega \in A \}] & = & 
   \sum_v \P [B \cap \{ \gamma_n = v \} \cap \{ \theta^n \omega \in A \}] \\
& = & \sum_v \int_{B \cap \{ \gamma_n = v \}} \P_{\pivot_n} [A] \, d\P  \\
& = & \int_B \P_{\pivot_n} [A] \, d\P \, . 
\end{eqnarray*}

We now fix $v$ and assume without loss of generality that 
$B \subseteq \{ \gamma_n = v \}$.  The identity~\eqref{eq:cond prob}
we need to prove now reduces to
\begin{equation} \label{eq:fix v}
\P [B \cap \{ \sigma^v \omega \in A \}] = \int_B \P_{\pivot(v)} [A] \, d\P 
\end{equation}
and we need to show it holds for all $B \in \BB_v^+$ where we recall 
$\BB_v^+$ denotes the $\sigma$-field
generated by $\BB_v^*$ together with $\pivot_v$ and $\sigma^v$ denotes \label{shift-def} shifting to $v$.

We claim it is enough to prove~\eqref{eq:fix v} for sets $B$
of the form $B_1 \times B_2$ where $B_2 = \{ \pivot (v) \leq b\}$
and $B_1$ is an element of $\BB_v^*$ contained in the event 
$\{ \dual_v \geq a \}$ for real numbers $0 < b < a < 1$. 
To see why this is enough, observe that the set of all $B$ 
for which~\eqref{eq:fix v} holds is a $\lambda$-system, 
meaning it is closed under increasing union and set theoretic 
difference of nested sets.  The class of sets of the form 
$B_1 \times B_2$ above are closed under intersection,
whence by Dynkin's Theorem~\cite[Theorem~2.1.2]{durrett4},
if~\eqref{eq:cond prob} holds when $B$ is in this class, 
then it holds for all $B$ in the $\sigma$-field generated
by this class, which is $\BB^+_v$.

Working on the left-hand side of~\eqref{eq:fix v},
\begin{eqnarray*}
\P [B \cap \{ \sigma^v \omega \in A \} ] & = & 
   \P [B_1 \cap \{ \sigma^v \omega \in A , \pivot (v) \leq b \}] \\
& = & \P [B_1] \P [\sigma^v \omega \in A , \pivot (v) (\omega) \leq b] \\
& = & \P [B_1] \P [\sigma^v \omega \in A \cap \{ \pivot_0 \leq b \}] \\
& = & \P [B_1] \P [A \cap \{ \pivot_0 \leq b \}] \, .
\end{eqnarray*}
Here, the first equality is definitional, the second uses independence
of $\BB_v^*$ and $\pivot (v)$, the third uses $\pivot (v) (\omega) \leq b$
if and only if $\pivot (\rtt) (\sigma^v \omega) \leq b$, and the last 
holds because $\sigma^v$ preserves the measure $\P$.

Working on the right-hand side of~\eqref{eq:fix v}, identify 
$(\Omega , \F , \P)$ as a product $(\Omega^{(1)} , \BB_v^* , \P^{(1)})
\times (\Omega^{(2)} , \CC_v , \P^{(2)})$ in the obvious way and compute
\begin{eqnarray*}
\int_B \P_{\pivot(v)} [A] \, d\P & = & \left ( \int_{B_1} 1 \, d\P_1 \right )
   \cdot \left ( \int_{\pivot(v) \leq b} \P_{\pivot(v)} [A] \, d\P_2 \right ) \\
& = & \P [B_1] \int \P_x [A \cap \{ \pivot_0 \leq b \} ] \, d\cL (x) \\
& = & \P [B_1] \P [A \cap \{ \pivot_0 \leq b \}] 
\end{eqnarray*}
because the $\P^{(2)}$-law of $\pivot(v)$ is $\cL$.  This finishes the proof.
$\Cox$
\end{proof}

\begin{pr} \label{pr:piv-dual-markov}
The sequence $\{ \pivot_n ,\dual_n\}$ is a time-homogeneous
Markov chain adapted to $\{ \BB_n^+ \}$ with initial distribution $\cL\times \delta_1$.
\end{pr}

\begin{proof} For any Borel $A\subseteq [0,1]\times [0,1]$ and any $x,y\in [0,1]$, define \begin{align*}
\tilde{A} &:=\left\{\omega\in\F :(\pivot_1(\omega),\dual_1(\omega))\in A\right\} \\
\tilde{A}_y &:=\left\{\omega\in\F :(\pivot_1(\omega),\dual_1(\omega)\wedge y)
\in A\right\} \\
\mathrm{and}\ \P_{x,y}[A]& :=\P_x[\tilde{A}_y]\,.
\end{align*}  
To show that $\left\{\pivot_n,\dual_n\right\}$ is a time homogeneous
Markov chain adapted to $\BB_n^+$, it will suffice to show that for any Borel  $A\subseteq [0,1]\times [0,1]$ and any $B \in\BB_n^+$,
\begin{equation}\label{JoshMark1}
\P\left[\left\{(\pivot_{n+1},\dual_{n+1})\in A\right\}\cap B\right]=\int_B\P_{\pivot_n,\dual_n}[A]\,d\P\,.
\end{equation}
Starting with the case where $A=[0,x]\times [0,y]$ for $x,y\in (0,1]$, we have \begin{align*}
\left\{(\pivot_{n+1},\dual_{n+1})\in A\right\}\cap B &=\left\{\pivot_{n+1}\leq x\right\}\cap\left\{\dual_{n+1}\leq y\right\}\cap B \\ 
&=\left\{\pivot_{n+1}\leq x\right\}\cap\left(\left\{\dual_n\leq y\right\}\cap B\right)\cup\left(\left\{\pivot_{n+1}\leq x\right\}\cap\left\{\dual_{\gamma_{n+1}}\leq y\right\}\right)\cap\left(\left\{\dual_n >y\right\}\cap B\right) \\
&=\left\{\omega :\theta^n\omega\in\tilde{A}'\right\}\cap\left(\left\{\dual_n\leq y\right\}\cap B\right)\cup\left\{\omega:\theta^n\omega\in\tilde{A}\right\}\cap\left(\left\{\dual_n>y\right\}\cap B\right)
\end{align*} where $A' :=[0,x]\times [0,1]$.  
It now follows from Proposition \ref{pr:markov} that \begin{align*}
\P\left[\left\{(\pivot_{n+1},\dual_{n+1})\in A\right\}\cap B\right] &=\int_{\left\{\dual_n\leq y\right\}\cap B}\P_{\pivot_n}[\tilde{A}']\,d\P +\int_{\left\{\dual_n > y\right\}\cap B}\P_{\pivot_n}[\tilde{A}]\,d\P \\
&=\int_B \P_{\pivot_n}[\tilde{A}']\one_{\dual_n<y}+\P_{\pivot_n}[\tilde{A}]\one_{\dual_n>y}\,d\P \\
&=\int_B\P_{\pivot_n,\dual_n}[A]\,d\P\,.
\end{align*}
Hence, \eqref{JoshMark1} has been proven for all $A$ of the form $[0,x]\times [0,y]$.  Since the set of all $A$ for which \eqref{JoshMark1} holds is a $\lambda$-system, and the collection of all sets of the form $[0,x]\times [0,y]$ is a $\pi$-system that generates the Borel sets in $[0,1]\times [0,1]$, it now follows from Dynkin's $\pi$-$\lambda$ Theorem that \eqref{JoshMark1} holds for all Borel sets $A$. $\Cox$
\end{proof}

\subsection{Analyzing the decay rate of $\left\{h_n\right\}$}
This section will be devoted to describing the limiting behaviour 
of the process $\left\{h_n\right\}$.

\begin{cor} \label{cor:markov}
~~\\[-4ex]
\begin{enumerate}[$(i)$]
\item 
The sequence $\{ \pivot_n := \pivot (\gamma_n) \}$ is a time-homogeneous
Markov chain adapted to $\{ \BB_n^+ \}$ with initial distribution $\cL$.  
\item 
Reparametrizing by letting $h_n := \pivot_n - p_c$,
a formula for the transition kernel of the chain $\{ h_n \}$ is given 
in terms of the OGF $\phi$ by $p(a , \cdot) = \mu_a$ where
$$\frac{d\mu_a}{dx} = C_a \delta_a + 
   \frac{\phi'\left(1-(p_c+a) g(p_c + a)\right)g'(p_c + x)}{g'(p_c + a)} 
   \one_{(0,a)} (x)$$
and 
$$C_a = 1 - \frac{\phi'\left(1-(p_c+a)g(p_c + a)\right)g(p_c + a)}{g'(p_c + a)} \, .$$
\end{enumerate}
\end{cor}

\begin{proof}
Conclusion $(i)$ follows from the recursion $\pivot_{n+1} 
= \pivot_n \circ \theta$ 
by applying the Markov property with
$n=1$ and $A$ of the form $\{ \pivot_1 \in S \}$.  Conclusion~$(ii)$
follows from the recursive description of $\pivot_n$ 
as the minimum of $\max \{ U (v) , \pivot (v) \}$ over children 
of $\gamma_{n-1}$ and $\gamma_{n}$ is the argmin.  More specifically, fix $0 < x < a$; then $$\P[\beta_1 \in [p_c + x,p_c + x+dx] \| \beta_0 \in [p_c + a,p_c + a + da]] = \frac{\P[\beta_1 \in [p_c + x,p_c + x + dx] \cap \beta_0 \in [p_c + a,p_c + a+da]]}{\P[\beta_0 \in [p_c + a,p_c + a+da]]}.$$

We note $\P[\beta_0 \in [p_c + a,p_c +a+da]] = g'(p_c + a)da + o(da)$.  To calculate the numerator, note that in order for this event to occur, up to a term of $O(dx^2)$, only one child of the root may have $[p_c + x,p_c + x+dx]$.  This child $v$ must have $U_v \in [p_c + a,p_c + a+da]$ and all other children must have pivot above $p_c + a + da$. This gives \begin{align*}\P[\beta_1 \in [p_c + x,p_c + x + dx]& \cap \beta_0 \in [p_c + a,p_c + a+da]] \\
&= \sum_{k = 1}^\infty \P[Z = k] k g'(p_c + x)dxda(1 - (p_c + a)g(p_c + a))^{k-1} + o(dxda) \\
&= \phi'(1 - (p_c + a)g(p_c + a))g'(p_c + x)\,dxda + o(dxda).
\end{align*}
Combining the two and taking $da \to 0^+$ completes the proof.
$\Cox$
\end{proof}

\begin{thm}\label{theorem:condist}
Let $U_0,U_1,\dots$ be a sequence of IID random variables each 
uniformly distributed on $(0,1)$, and let $M_n = \min
\left\{U_0,U_1,\dots,U_n\right\}$.  For each $C_1,C_2$ 
such that $0 < C_1 < p_c < C_2$, the process $\left\{h_n\right\}$ 
can be coupled with the process $\left\{M_n\right\}$ so that, with probability $1$, 
$h_n$ eventually (meaning for all sufficiently large $n$) 
satisfies $C_1 \cdot M_n \leq h_n \leq C_2 \cdot M_n$.
\end{thm}

\begin{proof}
We start by looking at the function
$$f_a(u) = \left\{\begin{array}{ll}0 
   &\text{if }u\geq a\\\frac{\phi'\left(1-(p_c+a) g(p_c + a)\right)g'(p_c + u)}
   {g'(p_c + a)} 
   &\text{otherwise }
\end{array}\right. \, .$$ 
Writing $u$ as $u=r\cdot a$ (for $r\in[0,1)$) and using 
Corollary \ref{cor:dersurvprob}, we find that
\begin{equation}\label{limeqprodlim}
\lim_{a \to 0}
   \frac{\phi'\left(1-(p_c+a)g(p_c + a)\right)g'(p_c + ra)}{g'(p_c + a)}
   = \lim_{a \to 0}\phi'\left(1-(p_c+a)g(p_C + a)\right) 
   \cdot\lim_{a \to 0} \frac{g'(p_c + ra)}{g'(p_c + a)}
   = \mu
\end{equation}
with the convergence clearly being uniform with respect to $r$.  
Turning now to the process $\left\{M_n\right\}$, if we define
$$\tilde{f}_a(u) = \left\{\begin{array}{ll}0 
   &\text{if }u\geq a\\1 
   &\text{otherwise }
\end{array}\right.$$ 
and the measures $\nu_a$, where $\nu_a(A)=(1-a)\one_{(a\in A)} + 
\int_0^1\tilde{f}_a(u)\one_{(u\in A)}\,du$, then we see that 
$\left\{M_n\right\}$ is a Markov chain with transition kernel $\tilde{p}(x,\cdot) = \nu_x(\cdot)$.

Note that \eqref{limeqprodlim} implies there must exist $\delta > 0$ 
such that for $a < \delta$ we have $\frac{1}{C_2} < f_a(u) < \frac{1}{C_1}$ 
on $(0,a)$.  Define $N_\delta:=\min \left \{ n:h_n < \delta\right\}$ and note that since $h_n\to 0$ a.s., it follows that 
$N_\delta < \infty$ a.s.  Define the family of functions $Q_r$ 
for $r\in[0,\delta)$ where $Q_r:[0,r)\rightarrow\R$ is defined as
\begin{equation}\label{qrform}
Q_r(x) = \sum_{j=0}^{\infty}q^{j+1} \left(C_2\cdot f_r\left(q^j x\right)-1\right)
\end{equation} 
where $q:= \frac{C_1}{C_2}$.  Observe that because $\frac{1}{C_2} < 
f_a(u) < \frac{1}{C_1}$ on $(0,a)$ for $a < \delta$, it follows that
\begin{equation}\label{qrbndblw}
Q_r(x) > \sum_{j=0}^{\infty}q^{j+1} \left(C_2\cdot C^{-1}_2-1\right) = 0
\end{equation}
and that
\begin{equation}\label{qrbndabv}
Q_r(x) < \sum_{j=0}^{\infty}q^{j+1} \left(C_2\cdot C^{-1}_1 -1\right)
   = \frac{q}{q-1}\cdot \frac{q-1}{q} =1 \, .
\end{equation}
In addition, note that it follows from \eqref{qrform} that we have
\begin{eqnarray}
\frac{1}{C_1}Q_r(x) + \frac{1}{C_2}\left(1-Q_r(qx)\right) & = & 
   C^{-1}_2 + \sum_{j=0}^{\infty}q^{j+1}
   \left(q^{-1}\cdot f_r\left(q^j x\right)-C^{-1}_1\right) 
   - \sum_{j=0}^{\infty}q^{j+1}\left(f_r\left(q^{j+1}x\right)-C^{-1}_2\right)
   \nonumber \\
& = & C^{-1}_2-\frac{C_1}{C_2-C_1} \cdot \left(C^{-1}_1-C^{-1}_2\right)
   + \sum_{j=0}^{\infty}q^j f_r\left(q^j x\right)
   - \sum_{j=1}^{\infty}q^j f_r\left(q^j x\right) \nonumber \\
& = & f_r(x) \,. \label{qrinddef}
\end{eqnarray}

We'll now use the family of functions $Q_r$, along with the process 
$\left\{h_n\right\}$ and the sequence $\{U_k\}$ defined in the statement 
of the theorem, to define a new sequence $\{V_k\}$.  
Letting $V_0=h_{N_\delta}$, we define $V_j$ for $j\geq 1$ as follows.  
First let $L_n=\min\left\{V_0,V_1,\dots,V_n\right\}$.  
Now if $C_1\cdot U_{N_\delta+j}\geq L_{j-1}$, set $V_j = C_1\cdot 
U_{N_\delta+j}$.  If instead $C_1\cdot U_{N_\delta+j}<L_{j-1}$, then with 
probability $Q_{L_{j-1}}(C_1\cdot U_{N_\delta+j})$ set $V_j$ equal to 
$C_1\cdot U_{N_\delta+j}$, and with probability $1 - Q_{L_{j-1}}
(C_1\cdot U_{N_\delta+j})$ set $V_j$ equal to $C_2\cdot U_{N_\delta+j}$.  
Next we define the process $\left\{\tilde{h}_n\right\}$ as
$$\tilde{h}_n = \left\{\begin{array}{ll} h_n &\text{if }n < N_\delta \\
   L_{(n-N_\delta)} 
   &\text{otherwise }\end{array}\right.$$
Observe that in order to show that $\{\tilde{h}_n\}$ 
has the same joint distribution as $\left\{h_n\right\}$, it will suffice 
to establish that for any $n>0$ and $0<x<y<r$ 
\begin{equation}\label{hneqhnt}
\P[\tilde{h}_{n+1}\in[x,y)\|\tilde{h}_n=r]=\P[h_{n+1}\in[x,y)\|h_n=r] =\int_x^y f_r(t)\,dt.
\end{equation}
In the case where $r\geq \delta$ we see that \eqref{hneqhnt} follows 
immediately from the definition of $\{\tilde{h}_n\}$.  
Alternatively, if $r<\delta$ then it will follow from the definition of 
$\{\tilde{h}_n\}$ that
\begin{equation}\label{rettosl}
\P[\tilde{h}_{n+1}\in[x,y)\|\tilde{h}_n=r] = 
   \int_x^y C^{-1}_1 Q_r(x)+C^{-1}_2 (1-Q_r(qx))\,dx =\int_x^y f_r(x)\, dx.
\end{equation}
  
Defining the times $\tau_1 = \min \left \{ n : V_n < V_0\right\}$, 
$\tau_2 = \min \left \{ n : U_{N_\delta+n} < M_{N_\delta} \right\}$, 
and $\tau = \max \{ \tau_1 , \tau_2 \}$, we see that $\tau_1 < \infty$ 
a.s. due to the fact that $\tilde{h}_n\rightarrow 0$ a.s. since 
$\left\{\tilde{h}_n\right\}$ has the same joint distribution as 
$\left\{h_n\right\}$, and $\tau_2<\infty$ a.s. due to the $U_j$'s 
being IID uniform on $(0,1)$.  Since $\tilde{h}_n = 
\min \left \{ V_1 , V_2 , \dots , V_{n-N_\delta} \right\}$ for 
$n\geq N_\delta + \tau$, $M_n = \min \left \{ U_{N_\delta + 1} , 
U_{N_\delta+2} , \dots , U_n \right \}$ for $n \geq N_\delta+\tau$, 
and $C_1U_{N_\delta+j} \leq V_j \leq C_2 U_{N_\delta+j}$ for all $j\geq 1$,
it can be concluded that $C_1  M_n \leq 
\tilde{h}_n\leq C_2  M_n$ for all 
$n\geq N_\delta+\tau$, 
thus establishing that $(\tilde{h}_n,M_n)$ gives us our 
desired coupling.  
$\Cox$
\end{proof}

This coupling is enough to prove convergence on the level of paths to the Poisson lower envelope process.  Let $\mathcal{P}$ be an intensity $1$ Poisson point process on the upper-half-plane; define the \emph{Poisson lower envelope process} by \begin{equation*}
    L(t) := \min\{y > 0 : (x,y) \in \mathcal{P} \text{ for some }x \in [0,t]\}.
\end{equation*}

Then we have \begin{cor} \label{cor:LPE}
For any $\ee > 0$, \begin{equation}
    (k h_{\lceil kt \rceil } / p_c )_{t \geq \ee} \overset{*}{\implies} (L(t))_{t \geq \ee}
\end{equation}
where $\overset{*}{\implies}$ denotes convergence in distribution of c\`adl\`ag paths in the Skorohod topology.
\end{cor}

\begin{proof}
We proceed in a similar fashion to \cite{angel-invasion2008}: we Poissonize, sandwich the Poissonized version of $k h_{\lceil kt \rceil}$ between two scaled copies of $L(t)$, and then use the strong law of large numbers to depoissonize.  

Consider an intensity $1$ Poisson process on $[0,\infty)$ and define $N(t)$ to be the number of
points in $[0,t]$.  Define $$\tilde{M}(t) = M_{N(t)}, \quad t\geq 0 $$
to be the Poissonized version of the min-uniform process  defined by $M_n =
\min\{U_1,\ldots,U_n\}$ for $n \geq 1$ and $M_0 = 1$.  Then note that both $\tilde{M}(t)$ and
$L(t)$ are continuous-time Markov processes that jump from height $z$ to height $z U[0,1]$ at
exponential rate $z$.  Moreover, the process $L_1(t) := 1 \wedge L(t)$ and $\tilde{M}(t)$ have the
same starting value and jump from $z$ to $z U[0,1]$ at exponential rate $z$.  Using the same
exponential clock and uniforms for both processes 
gives $$L_1(t) = \tilde{M}(t)$$ for all $t \geq 0$.  Since
$L(t)$ is eventually less than $1$, we have that there exists an almost-surely finite time
$\tau$ so that $$L(t) = \tilde{M}(t),\qquad t \geq \tau.$$

Thus, for all $k$, we have \begin{equation} \label{eq:Mt-L} k\tilde{M}(kt) = kL(kt) \overset{d}{=} L(t)
\end{equation} since for all $k$ the process $(kL(kt))_{t \geq 0}$ has the same law as $(L(t))_{t \geq 0}$.  

By the strong law of large numbers, for any fixed $\ee > 0$ and $\gamma > 0$, there exists an almost-surely finite random variable $K$ so that $$N(kt) \in [(1 - \gamma)kt,(1 +\gamma)kt],\qquad k \geq K$$ uniformly in $t \geq \ee.$  Since $\tilde{M}(t)$ is decreasing, this implies for all $k \geq K$ and uniformly in $t \geq \ee$ we have $$kM_{\lceil kt \rceil} \in [k\tilde{M}((1+\gamma)kt),k\tilde{M}((1-\gamma)kt)].$$  Combining this with equation \eqref{eq:Mt-L} gives 
\begin{equation} \label{eq:M-sandwich}
    \frac{1}{1 + \gamma}L'(t) \leq k M_{\lceil k t \rceil} \leq \frac{1}{1 - \gamma}L(t),\qquad k \geq K, ~t \geq \ee
\end{equation}
where $L'$ is a different version of $L$.

By Theorem \ref{theorem:condist}, for any $\delta > 0$, there is a coupling so that \begin{eqnarray} \label{eq:h-sandwich}
\frac{M_n}{1 + \delta} \leq h_n / p_c \leq (1 + \delta) M_n, \qquad n \geq n_0
\end{eqnarray}
where $n_0$ is an almost-surely finite stopping time.  Combining \eqref{eq:M-sandwich} and \eqref{eq:h-sandwich} gives $$ \frac{L'(t)}{(1 + \delta)(1 + \gamma)} \leq kh_{\lceil{kt\rceil}} / p_c \leq \frac{1+\delta}{1 - \gamma}L(t),\qquad k \geq \tilde{K},~ t\geq \ee$$ for some almost-surely finite $\tilde{K}$.  Taking $\delta, \gamma \downarrow 0$ completes the proof.  $\Cox$ 
\end{proof}

\begin{cor}\label{cor:contoexp}
The sequence $n\cdot h_n$ converges in distribution to $p_c\cdot\mathrm{exp}(1)$, where $\mathrm{exp}(1)$ is an exponential random variable with mean $1$.
\end{cor}

\begin{proof}
It suffices to show that for every $x\in(0,\infty)$, we have 
$\lim_{n \to\infty} \P[n\cdot h_n > x] = e^{-\mu x}$.  We know from Theorem \ref{theorem:condist} that
\begin{eqnarray}
\P[n\cdot h_n>x] & = & \P\left[h_n>\frac{x}{n}\right] \nonumber \\
& \geq & \P\left[C_1\cdot M_n > \frac{x}{n}\right] - \P[N_\delta+\tau>n] \nonumber \\
& = & \left(1-\frac{x/C_1}{n}\right)^n - \P[N_\delta+\tau>n] \,.\label{probhnbound}
\end{eqnarray}
Taking the $\liminf$ of the expressions on the right and left sides 
in \eqref{probhnbound}, while recalling that $N_\delta + \tau < \infty$ a.s., 
we find that $\liminf_{n\to\infty} \P[n\cdot h_n>x] \geq 
e^{-\frac{x}{C_1}}$.  Since $C_1<p_c$ is arbitrary, it then follows that 
$\liminf_{n\to\infty} \P[n \cdot h_n > x] \geq e^{-\mu x}$.  Conversely, 
Theorem \ref{theorem:condist} also implies that
\begin{eqnarray}
\P[n\cdot h_n>x] & = & \P\left[h_n>\frac{x}{n}\right] \nonumber \\
& \leq & \P\left[C_2\cdot M_n>\frac{x}{n}\right] + \P[N_\delta+\tau>n] \nonumber \\
& = & \left(1-\frac{x/C_2}{n}\right)^n+\P[N_\delta+\tau>n] .  \label{scndprobhnbnd}
\end{eqnarray}
Taking the $\limsup$ of the expressions on the right and left sides 
in \eqref{scndprobhnbnd} then gives $\limsup_{n\to\infty} 
\P[n\cdot h_n > x] \leq e^{-\frac{x}{C_2}}$ which, since $C_2>p_c$ 
is arbitrary, implies $\limsup\limits_{n\to\infty} \P[n\cdot h_n>x] 
\leq e^{-\mu x}$.  Combining this with the lower bound on the $\liminf$ gives $\lim_{n\to\infty} \P[n\cdot h_n > x]=e^{-\mu x}$.  $\Cox$
\end{proof}

\subsection{Decay of $h_n^*$}
\begin{pr} Define $h_n^* := \beta_n^* - p_c$ and $f(x) := \phi'(1 - (p_c + x)g(p_c + x))$.  Then $\{h_n,h_n^*\}$ has transition probabilities given by $p(\{a,b\}, \cdot ) = \nu_a \times \tilde{\nu}_{a,b}$ where \begin{align*}
    \frac{d \nu_a}{dx} &= \frac{f(a) g'(p_c + x)}{g'(p_c + a)} \one_{x < a} + C_a \delta_a \\
    \text{and}\quad \frac{d \tilde{\nu}_{a,b}}{dx} &= -\frac{f'(x)}{f(a)} \one_{a < x < b} + \tilde{C}_{a,b}\delta_b
\end{align*} 
with $C_a = f(a)(p_c + a)$ and $\tilde{C}_{a,b} = \frac{f(b)}{f(a)}$.
\end{pr}
\begin{proof}
$(i)$ Since $h_{n+1}^* = \min\{\beta_n^*-p_c,\beta_{\gamma_{n+1},\gamma_n}^* - p_c\}$, it follows that \begin{align*}
    \P[h_{n+1} \in (a, a+ da),\,& h_{n+1}^* \in (b,b+db) \| h_n \in (a,a + da), h_n^* \in (b,b + db)]  \\
    &= \P[\beta_{n+1} \in (p_c + a, p_c + a + da), \beta_{\gamma_{n+1},\gamma_n}^* > p_c + b \| \beta_n \in (p_c + a, p_c + a + da) ] \\
    &= \P[\beta_1 \in (p_c + a,p_c + a + da), \beta_{\gamma_1,\rtt}^* > p_c + b \| \beta_0 \in (p_c + a, p_c + a + da)] \\ 
    &= \frac{\left(\sum_{k = 1}^\infty \P[Z = k] k (p_c + a) g'(p_c + a)(1 - (p_c + b)g(p_c + b))^{k-1}\,da \right) + o(da)}{g'(p_c + a)\,da + o(da)} \\
    &= (p_c + a) \phi'(1 - (p_c + b)g(p_c + b)) + o(1) \\
    &= (p_c + a) f(b) + o(1).
\end{align*}

$(ii)$ For $a < x < b$ we have \begin{align*}
    \P[h_{n+1} \in& (a, a+ da), h_{n+1}^* \in (x,x + dx) \| h_n \in (a,a + da), h_n^*(b, b + db)] \\
    &= \P[\beta_1 \in (p_c + a, p_c + a + da), \beta_{\gamma_1,\rtt}^* \in (p_c + x, p_c + x + dx) \| \beta_0 \in (p_c + a, p_c + a + da)] \\
    &= \frac{\left( \sum_{k=2}^\infty \P[Z = k] k (p_c + a) g'(p_c + a)\left(-  \frac{d}{dx}\left(1 - (p_c + x)g(p_c + x)\right)^{k-1} \right)dx\,da \right) + o(dx\,da) }{g'(p_c + a)da + o(da)} \\
    &= (p_c + a) (g(p_c + x) + (p_c + x)g'(p_c + x)) \phi''(1 - (p_c + x)g(p_c + x))\,dx + o(dx)\\
    &= (p_c + a) f'(x) \,dx + o(dx).
\end{align*}

$(iii)$ For $z < a < b$ we have \begin{align*}
    \P[h_{n+1} \in (z,z+ dz),h_{n+1}^* &\in (b,b + db) \| h_n \in (a,a + da), h_n^* \in (b,b + db) ]  \\ &= \P[\beta_1 \in (p_c + z, p_c + z + dz), \beta_{\gamma_1,\rtt}^* > b \| \beta_0 \in (p_c + a, p_c + a + da)] \\ 
    &= \frac{\left(\sum_{k=1}^\infty \P[Z = k] k g'(p_c + z) (1 - (p_c + b)g(p_c + b))^{k-1}\,dz\,da \right) + o(dz\,da)}{g'(p_c + a) \,da + o(da)} \\ 
    &= \frac{g'(p_c + z)}{g'(p_c + a)}\phi'(1 - (p_c + b)g(p_c + b))\,dz + o(dz) \\
    &= \frac{g'(p_c + z)}{g'(p_c + a)}f(b) \,dz + o(dz).
\end{align*} 
$(iv)$ For $z < a < x < b$, we have \begin{align*}
    \P[h_{n+1} \in& (z,z + dz), h_{n+1}^* \in (x,x + dx) \| h_n \in (a, a + da), h_n^* \in (b, b + da)] \\
    &= \P[\beta_1 \in (p_c + z, p_c + z + dz), \beta_{\gamma_1,\rtt}^* \in (p_c + x, p_c + x + dx) \| \beta_0 \in (p_c + a, p_c + a + da)] \\
    &= \frac{\left(\sum_{k = 2}^\infty \P[Z = k] k g'(p_c + z)\left(-\frac{d}{dz}\left(1 - (p_c + x)g(p_c + x) \right)^{k-1} \right)\,da\,dx\,dz \right) + o(da\,dx\,dz)}{g'(p_c + a)\,da + o(da)} \\
    &= \frac{g'(p_c + z)}{g'(p_c + a)} \left(- \frac{d}{dx}(\phi'(1 - (p_c + x)g(p_x + x))) \right)\,dx\,dz + o(dx\,dz) \\ 
    &= -\frac{g'(p_c + z)}{g'(p_c + a)} f'(x) \,dx\,dz + o(dx\,dz)\,.
\end{align*}

Noting that $h_n < h_n^*$, it follows from $(i), (ii), (iii),$ and $(iv)$ that \begin{align*}
    p\left(\{a,b\}, \{A,B\} \right) &= (p_c + a)\one_{a \in A} \cdot f(b) \one_{b \in B} \\
    &\quad + \left(\int_B \one_{a < x < b} f'(x) \,dx \right)(p_c + a) \one_{a \in A} + \left( \int_A \one_{z < a} \frac{g'(p_c + z)}{g'(p_c + a)}\,dz \right) f(b) \one_{b \in B} \\
    &\quad + \int_{A \times B} \one_{a < x < b} f'(x) \one_{z < a} \frac{g'(p_c + z)}{g'(p_c + a)}\,dx \,dz\,.
\end{align*}
Hence, we see that $p(\{a,b\}, \cdot) = \mu_a \times \tilde{\mu}_{a,b}$ where \begin{align*}
    \frac{d \mu_a}{dx} &= \one_{x < a} \frac{g'(p_c + x)}{g'(p_c + a)} + (p_c + a) \delta_a \quad \textrm{and} \\
    \frac{d \tilde{\mu}_{a,b}}{dx} &= \one_{a < x < b} \cdot f'(x) + f(b) \delta_b\,.
\end{align*}

Noting that $\mu_a((0,a]) = \frac{g(p_c + a)}{g'(p_c + a)} + (p_c + a) = \frac{1}{f(a)}$ and $\tilde{\mu}_{a,b}((a,b]) = f(a)$, we define probability measure $\nu_a = f(a) \mu_a$ and $\tilde{\nu}_{a,b} = \frac{1}{f(a)} \tilde{\mu}_{a,b},$ and we see that $\nu_a$ and $\tilde{\nu}_{a,b}$ satisfy the statement in the proposition. $\Cox$
\end{proof}

\begin{thm} \label{th:dual-piv-decay}
There exists $C>0$ such that for any $t \in(1/2,1)$, $\P[h_n^*>n^{-t}]$ is $O(e^{-Cn^{1-t}})$.
\end{thm}
\begin{proof}
We start by looking at $\frac{d \nu_a}{dx}$.  First we want to show that $\frac{f(a)}{g'(p_c+a)}$ is bounded below by something positive.  Noting that $1-g(p)$ is the unique non trivial fixed point of $\phi_p(x):=\phi(px+1-p)$, it follows
from implicit differentiation that $$g'(p_c+x)=\frac{g(p_c+x)\phi'(1-(p_c+x)g(p_c+x))}{1-(p_c+x)\phi'(1-(p_c+x)g(p_c+x))}$$which then implies that
\begin{equation}\label{ratifgp}\frac{f(a)}{g'(p_c+a)}=\frac{1-(p_c+a)\phi'(1-(p_c+a)g(p_c+a))}{g(p_c+a)}\,.\end{equation}As $a\to 0$, this expression equals
\begin{align*}\frac{1-(p_c+a)(\mu-\phi''(1)(p_c+a)g(p_c+a)+o(g(p_c+a)))}{g(p_c+a)} &= \frac{p_c\phi''(1)(p_c+a)g(p_c+a)-a\mu+o(a)}{g(p_c+a)}\\ &= p_c^2\phi''(1)-\frac{\mu}{g'(p_c)}+o(1) \\ &= p_c^2\phi''(1)-\frac{p_c^2\phi''(1)}{2}+o(1) \\ &= \frac{p_c^2\phi''(1)}{2}+o(1)\,.
\end{align*}Hence, there must exist some $r'>0$ such that if $a<r'$ then \eqref{ratifgp} is greater than $\frac{p_c^2\phi''(1)}{3}$.  Next observe that the numerator of \eqref{ratifgp} is equal to one minus the derivative of $\phi_p(x)$ evaluated at the fixed point $1-g(p)$ where $p=p_c+a$) from which it follows that this numerator, and therefore \eqref{ratifgp} itself, is positive whenever $a>0$.  Since we can also see that \eqref{ratifgp} is continuous on the compact set $[r',1-p_c]$, it follows that \eqref{ratifgp} must be bounded below by some value $C'>0$ on $[r',1-p_c]$.  Now setting $C''=\text{min}\{C',\frac{p_c^2\phi''(1)}{3}\}$, we find that \eqref{ratifgp} is greater than or equal to $C''$ on $[0,1-p_c]$.  Finally, if we couple this with the fact that $g'(p_c+x)\to\frac{2}{p_c^3\phi''(1)}$ as $x\to 0$, which in turn implies that $\exists\ r>0$ such that $g'(p_c+x)>\frac{1}{p_c^3\phi''(1)}$ on $[0,r]$, we find that if we set $\tilde{C}=\frac{C''}{p_c^3\phi''(1)}$, then for any $x,a$ where $x<r$ and $x<a$ we have $\frac{d \nu_a}{dx}\geq\tilde{C}$.

Now turning our focus towards $\frac{d \tilde{\nu}_{a,b}}{dx}$, observe that in the expression $\frac{-f'(x)}{f(a)}$, the numerator goes to $2\mu^2$ as $x\to 0$, which can be seen by differentiating and noting that $g'(p_c)=\frac{2}{p_c^3\phi''(1)}$; additionally, the denominator goes to $\mu$ as $a\to 0$.  Hence, the ratio goes to $2\mu$ as $x\to 0,a\to 0$.  From this it follows that $\exists\ \ell >0$ such that if $a<x<\ell$ and $x<b$ then $\frac{d \tilde{\nu}_{a,b}}{dx}>\mu$.  Next we note the following string of inequalities.\begin{eqnarray}
\P[h_n^*\geq n^{-t}] & = & \P\left[h_{\lfloor\frac{n}{2}\rfloor}\geq\frac{n^{-t}}{2}\right]\P\left[h_n^*\geq n^{-t}\,\Big|\,h_{\lfloor\frac{n}{2}\rfloor}\geq\frac{n^{-t}}{2}\right]+\P\left[h_{\lfloor\frac{n}{2}\rfloor}<\frac{n^{-t}}{2}\right]\P\left[h_n^*\geq n^{-t}\,\Big|\,h_{\lfloor\frac{n}{2}\rfloor}<\frac{n^{-t}}{2}\right]\nonumber \\
& \leq & \P\left[h_{\lfloor\frac{n}{2}\rfloor}\geq\frac{n^{-t}}{2}\right] + \P\left[h_n^*\geq n^{-t}\,\Big|\,h_{\lfloor\frac{n}{2}\rfloor}<\frac{n^{-t}}{2}\right] \nonumber \\
& \leq & \P\left[h_{\lfloor\frac{n}{2}\rfloor}\geq\frac{n^{-t}}{2}\right]+\prod_{j=\lceil\frac{n}{2}\rceil}^n \P\left[h_j^*\geq n^{-t}\,\Big|\,h_{\lfloor\frac{n}{2}\rfloor}<\frac{n^{-t}}{2},h_{j-1}^*\geq n^{-t}\right] .  \label{ineqinvhnst}
\end{eqnarray}Now using \eqref{ineqinvhnst}, along with the results from the previous paragraph, we find that if $\frac{n^{-t}}{2}<r$ and $n^{-t}<\ell$, then $\P[h_n^*\geq n^{-t}]\leq\left(1-\frac{\tilde{C}}{2}n^{-t}\right)^{\lfloor\frac{n}{2}\rfloor}+\left(1-\frac{\mu}{2}n^{-t}\right)^{\lceil\frac{n}{2}\rceil}$.  Defining $C=\min\left\{\frac{\tilde{C}}{4},\frac{\mu}{4}\right\}$, we finally get that $\P[h_n^*\geq n^{-t}]$ is $O\left(e^{-Cn^{1-t}}\right)$, thus completing the proof. $\Cox$
\end{proof}

\section{Proof of Theorem~\protect{\ref{th:main}}} \label{sec:rest}

This section is devoted to the proof of Theorem \ref{th:main}.  A high-level summary is given:

\begin{outline} \label{proof-outline}
To prove Theorem \ref{th:main}, we utilize Corollary \ref{cor:X} for a suitable choice of $Y(v)$: \begin{enumerate}
    \item Define $$Y(v) := X(v)\one_{A_v}$$ where $A_v :=\left\{|q(w)/p(w)-1|<Mn^{-t}\ \text{for each child }w\ \text{of }v\right\}$ for some $t>1/2$ and $M>0$.  Then by Proposition \ref{pr:KL}, $Y(\gamma_n)$ is summable so by \ref{cor:X}, it suffices to show that, with probability $1$, $|q(w) / p(w) - 1| \geq Mn^{-t}$ for only finitely many $w$ that are children of $\gamma_n$.
    \item For a vertex $v \neq \rtt$ and $p > p_c$, define \begin{equation}
        \qt(v,p) := \QQ_{p}[\sigma_{\Par{v}}^{-1} \{v = \gamma_1 \}]. \label{qt-def}
    \end{equation}
    In words, $\qt(v,p)$ considers the tree rooted at $\Par{v}$ and finds the probability that $v$ is in the backbone conditioned on the root having pivot at most $p.$  We then have $q(v) = \E_*^{(n)}[\qt(v,\beta_n^*)]$\label{excondbth}, where $\beta_n^*$ is as defined in Definition \ref{def:dual-piv} and $\E_*^{(n)}:=\E[\cdot|\mathcal{T},\gamma_n]$.
    \item Show that $\qt(v,\beta_n^*)$ is close to $\frac{g(T(v),\beta_n^*) }{\sum_w g(T(w), \beta_n^*) } $ with high probability (where the sum in the denominator is over all children of $\gamma_n$ including $v$).
    \item Bound $$\frac{\frac{g(T(v),\beta_n^*)}{\sum_w g(T(w),\beta_n^*)}}{p(v)} - 1.$$ with high probability.
\end{enumerate}
\end{outline}

\subsection{Comparing $\qt$ and the ratio of survival functions}

The goal of this section is to accomplish step $3$ of Outline \ref{proof-outline}.  This takes the form of \begin{lem}\label{lem:qt-ratio-survival} Let $\{w_k\}_{k=1}^d$ be an enumeration of the children of $v$.  Then for any $p > p_c$ and $j$, \begin{equation}
    \left|\qt(w_j,p) - \frac{g(T(w_j),p)}{\sum_{k=1}^d g(T(w_k),p)}\right| \leq \frac{g(T(v),p)}{1 - g(T(v),p)}\cdot \frac{g(T(w_j),p)}{\sum_{k=1}^d g(T(w_k),p)}.
\end{equation}
\end{lem}
\begin{proof}
Define $$A_j = \qt(w_j,p) - \frac{g(T(w_j),p)}{\sum_{k=1}^d g(T(w_k),p)}$$ 
and write 
\begin{eqnarray}
\qt(w_j,p) & = & 
   \frac{\P_*[U_{w_j} \vee \beta(w_j)\text{ is smallest } \| \beta(v) \leq p]}
   {\sum_{i = 1}^d \P_*[U_{w_i} \vee \beta(w_i)\text{ is smallest }\| \beta(v) 
   \leq p]} \nonumber \nonumber \\[2ex]
& = & \frac{\P_*[U_{w_j} \vee \beta(w_j)\text{ smallest and } 
  U_{w_j} \vee \beta(w_j) \leq p]}
   {\sum_{i = 1}^d \P_*[U_{w_i} \vee \beta(w_i)\text{ smallest and } 
  U_{w_i} \vee \beta(w_i) \leq p]} \, .
\end{eqnarray}

For each $j$, we observe that 
\begin{eqnarray*}
p\cdot g(T(w_j),p)(1- B_j) & \leq & 
   \P_*[U_{w_j} \vee \beta(w_j)\text{ smallest and } U_{w_j} \vee \beta(w_j) \leq p] \\[2ex]
& \leq & p \cdot g(T(w_j),p)
\end{eqnarray*}
where $1 - B_j = \prod_{1\leq i \neq j \leq d} (1 - p g(T(w_i),p))$.  
The upper bound is the probability that $U_{w_j}\vee\beta(w_j)\leq p$, while the lower bound is the probility that $U_{w_j}\vee\beta(w_j)\leq p$, and that this does not hold for any of the siblings of $w_j$. 

This gives the bounds 
\begin{equation} \label{eq:qt-sandwich}
\frac{g(T(w_j),p)(1 - B_j)}{\sum_{k=1}^d g(T(w_k),p)} 
   \leq \qt(w_j,p) 
   \leq \frac{g(T(w_j),p)}{\sum_{k=1}^d g(T(w_k),p) (1 - B_k) }.
\end{equation} 

Sandwich bounds on the difference with survival ratios follow: 
\begin{align}
\frac{-B_j g(T(w_j,p)}{\sum_{k=1}^d g(T(w_k),p)} 
   \leq  A_j \leq  \frac{\sum_{k=1}^d [g(T(w_k),p) g(T(w_j),p) B_k]}
   {(\sum_{k=1}^d g(T(w_k),p))(\sum_{k=1}^d [g(T(w_k),p)(1 - B_k)])} \,.
   \label{eq:Aj-with-B-sandwich}
\end{align}

Finally, the simple bound of 
$$B_k \leq 1 - \prod_{i=1}^d(1 - pg(T(w_i),p)) 
   = g(T(v),p)$$
allows us to rewrite equation \eqref{eq:Aj-with-B-sandwich} as 
\begin{equation} \label{eq:Aj-sandwich}
- \frac{g(T(v),p)g(T(w_j),p)}{\sum_{k=1}^d g(T(w_k),p)}
   \leq A_j \leq \frac{g(T(v),p)}{1 - g(T(v),p)}
   \frac{g(T(w_j),p)}{\sum_{k=1}^d g(T(w_k),p)}  \; .
\end{equation}
$\Cox$
\end{proof}

\subsection{Bounds on $E(v,\ee)$} \label{sec:bounds-E}

Before completing the final step in the proof of Theorem \ref{th:main}, we require estimates on $g$.
For a fixed vertex $v$ in a tree $T$ define $E(v,\ee)$ by 
$$g(T(v),p_c + \ee) = g(p_c + \ee)\left(W(v) + E(v,\ee) \right).$$

\begin{pr}\label{pr:E-bound-1} Suppose the offspring distribution of $Z$ has
$p\geq 2$ moments.  
Then for any $\delta,\ell$ for which both $0<\delta<1$ and $0<\ell<\frac{1}{2}$, there exist constants $C_i > 0$ so that for all $\ee$ sufficiently small, a fresh Galton-Watson tree rooted at $v$ satisfies
\begin{equation} \label{eq:E-bound}
|E(v,\ee)| \leq C_1 W(v) \ee^{1 - \delta} 
+  C_2\ee^{1-2\ell}\sum_{j = 1}^{\lceil \ee^{-\delta}\rceil -1} W_j(v)
\end{equation}
 with probability at least $1 - C_3 \ee^{p\ell -\delta}$.
\end{pr}

\begin{proof} Let $c_1,c_2,c_3$ be the constants from Lemma \ref{lem:mean}, and fix $\delta > 0$.  Then for $m = \lceil \ee^{-\delta} \rceil $, we have 
\begin{equation} \label{eq:gm-squeeze}
|g_m(\TT(v),p_c + \ee) - g(\TT(v),p_c + \ee)| < c_1 e^{-c_2 / \ee^{\delta}}
\end{equation} 
with probability at least $1 - e^{-c_3 / \ee^{\delta } }$, which implies that \eqref{eq:gm-squeeze} holds for the 
root and all children of the root with 
probability at least $1 - (\mu+1)e^{-c_3 / \ee^\delta}$.  Utilizing \eqref{eq:gm-squeeze} and the fact that $g(p_c + \ee) = \Theta(\ee)$ as $\ee \to 0^+$ (while also making sure to select $c_3<c_2$) gives
\begin{equation}
\frac{1}{g(p_c + \ee)}\left|g_m(\TT(v),p_c + \ee) - g(\TT(v),p_c + \ee)\right| < 
c_1\frac{1}{g(p_c + \ee)}e^{-c_2 / \ee^\delta} \lesssim e^{-c_3 / \ee^\delta }\,.
\end{equation}  
By \cite{dubuc}, there exist positive constants $C_1'$ and $c_2'$ so that $$\P[W \leq a] \leq C_1' a^{c_2'}.$$  This implies that $C_1 e^{-c_3 / \ee^\delta} \leq W \ee^{1 - \delta}$ with probability at least $1 - C e^{-c / \ee^{\delta}}$ for some new constants.  Thus,
to show equation 
$\eqref{eq:E-bound}$, it is sufficient to examine $g_m(\TT(v),p_c + \ee)$.

The Bonferroni inequalities imply that 
$$\FO_m(v,\ee) - \SO_m(v,\ee) \leq g_m(\TT(v),p_c +\ee) 
\leq \FO_m(v,\ee) $$

where \begin{align*}\FO_m(v,\ee) &:= \left( 1 + 
\frac{\ee}{p_c}\right)^m W_m(v) g(p_c + \ee) \\ \text{ and }  
\SO_m(v,\ee) &:= g(p_c + \ee)^2
    \sum\limits_{\substack{u,w \in \TT_m(v) \\ u \neq w}}  
    (p_c + \ee)^{2m-|u\wedge w|}.
\end{align*}

To bound $g_m(\TT(v),p_c + \ee) - g(p_c + \ee)W(v)$, we first bound 
$\frac{\FO_m(v,\ee)}{g(p_c + \ee)} - W(v)$.  Write 
$$\frac{\FO_m(v,\ee)}{g(p_c + \ee)} - W(v) 
= W(v) \left(\left(1 + \frac{\ee}{p_c} \right)^m - 1\right) 
+ [W_m(v) - W(v)](1 + \ee/p_c)^m.$$

Note first that $\left|(1 + \ee / p_c)^m - 1\right| \leq C m \ee / p_c$ for some $C>0$.  
Recalling that $m = \lceil \ee^{-\delta} \rceil $ gives a bound of $C\ee^{1 - 
\delta}.$  Additionally, we have $(1  + 
\ee/p_c)^m \leq 2$ for $\ee$ sufficiently small.  We now look towards $|W_m(v) - W(v)|$.

By \cite[Chapter~I.13]{athreya-ney}, we have that 
$$\Var[W_m(v) - W(v) \| W_m(v)] 
= \frac{W_m(v)}{\mu^m} \left(\frac{\Var[Z]}{\mu^2 - \mu} \right).$$

By the law of total variance, this implies that 
$$\Var[W_m(v) - W(v) ] = \frac{1}{\mu^m}\frac{\Var[Z]}{\mu^2 - \mu} =: 
\frac{C_Z}{\mu^m}.$$  Chebyshev's inequality then gives 
$$\P[|W_m(v) - W(v)| > \mu^{-m/3}] \leq C_z \mu^{-m/3}.$$

Since $\mu^{-m/3} \leq \mu^{-\ee^{-\delta}/3} \leq C_2 e^{-c_1 / \ee^{c_2}}$ 
for some positive constants $C_2$ and $c_1,c_2$, we have that 

\begin{equation}\label{eq:FO-bound1}
\frac{\left|\FO_m(v,\ee) - g(p_c + \ee)W(v)\right|}{g(p_c + \ee)} 
\leq C_1 W(v) \ee^{1 - \delta} +  C_2 e^{-c_1 / \ee^{c_2}}
\end{equation} 
with probability at least $1 - C_Z \mu^{-m / 3} = 1 - C_3 e^{-c_3 / \ee^{c_4}}$.

By computing the lower probabilities of $W$ again, recall that there exist constants $C_1'$ and $c_2'$ so that $$\P[W \leq a] \leq C_1' a^{c_2'}.$$
This implies that $C_2 e^{-c_1 / \ee^{c_2}} < C_1 W(v) \ee^{1-\delta}$ with probability at least $1 - C e^{-c_2'c_1 / \ee ^{c_2} }$.  Relabeling constants, this means that for sufficiently small $\ee$, we can upgrade \eqref{eq:FO-bound1} to \begin{equation}\label{eq:FO-bound}
\frac{\left|\FO_m(v,\ee) - g(p_c + \ee)W(v)\right|}{g(p_c + \ee)} 
\leq C_1 W(v) \ee^{1 - \delta}
\end{equation}
with probability at least $1-e^{-c_1 / \ee^{c_2}}$.

The last piece is to bound $\SO_m(v,\ee)/g(p_c + \ee).$  By Fubini's theorem, 

\begin{align*}
    \frac{\SO_m(v,\ee)}{g(p_c + \ee)} &= g(p_c + \ee) \sum\limits_{\substack{u,w \in \TT_m(v) \\ u \neq w}}(p_c + \ee)^{2m - |u \wedge w|} \\
    &\leq 2 g(p_c + \ee) \sum_{j = 0}^{m-1} p_c^{2m - j} \sum_{u,v : |u \wedge v| = j} 1 \\
    &\leq 2g(p_c + \ee) \sum_{j = 0}^{m-1} p_c^j\sum_{u \in T_j} \sum_{1 \leq i < k} W_{m-j-1}^{(i)}(u)W_{m - j-1}^{(k)}(u) \\
    &\leq g(p_c + \ee) \sum_{j = 0}^{m-1} p_c^j \sum_{u \in T_j} W_{m-j}(u)^2  
\end{align*}
where the second inequality is from the bound $(1 + \frac{\ee}{p_c})^{2m} \leq 2$ for sufficiently small $\ee$.

Note that for each $j$ the innermost sum is a sum of IID random variables.  
We utilize the Fuk-Nagaev inequality from \cite{fuk-nagaev} which states 
$$\P\left[\sum_{u \in T_j} [W_{m-j}(u)^2 -\E W_{m-j}^2] > t \,\bigg|\, Z_j\right] 
\leq C_{p}t^{-p/2}Z_j^{p/4} + \exp\left(-C\frac{t^2}{Z_j} \right).$$

Applying this bound for $t = \E W_{m-j}^2 Z_j\ee^{-2\ell}$ gives 
\begin{align*}\P\left[\sum_{u \in T_j} [W_{m-j}(u)^2 - \E W_{m-j}^2] 
> (\E W_{m-j}^2 )Z_j\ee^{-2\ell} \,\bigg|\, Z_j \right] 
&\leq C_p' \ee^{p\ell}(Z_j)^{-p/4} + \exp\left(-C' Z_j/ \ee^{4\ell} \right) \\
&\leq C_p'' \ee^{p\ell}
\end{align*}
for some choice of $C_p'' > C_p'.$  By applying this bound and a union bound, we get 
$$\frac{\SO_m(v,\ee)}{g(p_c + \ee)} \leq g(p_c + \ee)(1 + \ee^{-2\ell}) 
\sum_{j = 0}^{m-1}\left(\E W_{m-j}^2 \right)Z_j(v) p_c^j \leq C g(p_c + \ee) \ee^{-2\ell} 
\sum_{j = 0}^{m-1} W_j(v) $$
with probability at least $1 - C_p''m\ee^{p\ell}$ for some new choice of $C$. 
This means that for a fresh Galton-Watson tree, 
$$\P\left[ \frac{\SO_m(v,\ee)}{g(p_c + \ee)} > C g(p_c + \ee) \ee^{-2\ell} 
\sum_{j = 0}^{m-1} W_j(v) \right] \leq m C_p''\ee^{p\ell}.$$

Recalling that $g(p_c + \ee) = 
\Theta(\ee)$ now gives  
$$\SO_m(v,\ee) \leq C_2 
\ee^{2-2\ell}\sum_{j = 0}^{m-1} W_j(v)$$ 
with probability at least $1 - C \ee^{p\ell -\delta}$ for some new $C$.  Along with equations 
\eqref{eq:gm-squeeze} and \eqref{eq:FO-bound}, this now implies the proposition. $\Cox$
\end{proof}

\begin{cor}\label{cor:E-bound-2}
Suppose the offspring distribution of $Z$ has $p > 1$ moments and $p_1 := \P[Z = 1]$.  Let $\delta,\ell ,d$ be positive constants such that \begin{equation} \label{eq:alpha-definition}
\alpha = 1 - 3\ell-(1+d)\delta
\end{equation}is greater than $\frac{1}{2}$.  Then there exists a constant $C>0$ such that for all $\ee > 0 $ sufficiently small 
\begin{equation} \label{eq:nice-E-bound}
    |E(v,\ee)| \leq C W(v) \ee^{\alpha}
\end{equation}
for the root and its children with probability at least $1 - C \ee^{\delta'}$ for $\delta'=\text{min}\left\{p\ell -\delta,\frac{\log(1/p_1)}{\log(\mu)}d\delta\right\}$.
\end{cor}
\begin{proof}
The first term in equation \eqref{eq:E-bound}
is always eventually smaller than $W(v) \ee^{\alpha}$ since the 
exponent on $\ee$ is larger.  The final term in equation \eqref{eq:E-bound} can now
be dealt with separately.

By \cite[Theorems 0 and 5]{bingham-doney74}, if $Z$ is 
in $L^p$, then $W_k \xrightarrow{L^p} W$, implying $\E[|W_k - W|^p] \leq C$ for 
some $C > 0$.  Therefore, $$\P[|W_k - W| > \ee^{-\ell}] \leq 
C \ee^{p\ell}.$$

Conditioning on $Z_1$, applying a union bound, and taking expectation implies that $$\sum W_k \leq m(\ee^{-\ell} + W)$$ for the root and all of its children with probability at least $1 - C(1 + \mu) \ee^{p\ell -\delta}.$  Recalling $m = \lceil \ee^{-\delta} \rceil$ and applying this to the latter term in equation \eqref{eq:E-bound} gives  $$|E(v,\ee)| \leq C_1 
W(v) \ee^{1 - 2\ell - \delta} + C_2 \ee^{1 - 3\ell - \delta}$$
with probability at least $1 - C \ee^{p\ell -\delta}$.

In the case where $p_1 = 0$, the lower tails on $W$ provided by \cite{dubuc} show that for any $r_1,r_2>0$ we have $\P[W(v) < \ee^{r_1}] = o(\ee^{r_2})$, thereby showing $W(v) < \ee^{r_1}$ with probability less than $\ee^{r_2}$ for $\ee$ sufficiently small.  Setting $r_1=d\delta$ and $r_2=p\ell -\delta$ completes the proof when $p_1 = 0.$

When $p_1 > 0$, there exists a constant $C$ so that for all $a \in (0,1)$ $$\P[W < a] \leq C a^{\log(1/p_1) / \log(\mu)}.$$  This implies that for $\alpha$ as in \eqref{eq:alpha-definition}, \begin{equation}\label{eq:W-lower-tail}
\P[W(v) < \ee^{1 - 3\ell - \delta - \alpha}] \lesssim \ee^{\frac{\log(1/p_1)}{\log(\mu)}d\delta}.
\end{equation} 
Performing a union bound for the root and all of its children again completes the proof.
$\Cox$
\end{proof}

\subsection{Completing the Argument}
With Corollary \ref{cor:E-bound-2} in place, we're ready to complete the proof of Theorem \ref{th:main}.  Recalling step $1$ of Outline \ref{proof-outline}, it will be sufficient to establish the following proposition.

\begin{pr} \label{pr:finalnecpro}
Letting $q:=\frac{\log(\mu)}{\log(1/p_1)}$, if \begin{equation}\label{lastineneeded}2p^2 q^2+(3p^2+5p)q+(-p^2+11p-4)<0,\end{equation}then there exists $M>0$ and $t \in (1/2,1)$ such that, with probability $1$, $\left|\frac{q(v)}{p(v)}-1\right|>3Mn^{-t}$ for only finitely many children of the backbone.
\end{pr}

\begin{proof}
We start by noting that for a tree $T$ and vertex $v$, $\left|\frac{q(v)}{p(v)}-1\right|=\left|\E^{(n)}_*\left[\frac{\tilde{q}(v,\beta_n^*)}{p(v)}-1\right]\right|$.  Now define $$A_n :=\left\{h_n^*\leq n^{-\frac{t}{a}}\right\}\bigcap\limits_{v:\Par{v}=\gamma_n}\left\{\frac{1}{p(v)}\left|\tilde{q}(v,\beta_n^*)-\frac{g(T(v),\beta_n^*)}{\sum\limits_{w : \Par{w} = \gamma_n} g(T(w),\beta_n^*)}\right|\leq 3Cn^{-t}\right\}\bigcap\limits_{v:\Par{v}=\gamma_n}\left\{\left|\frac{\frac{g(T(v),\beta_n^*)}{\sum g(T(w),\beta_n^*)}}{p(v)}-1\right|\leq 3n^{-t}\right\}$$where $C$ is as in Corollary \ref{cor:E-bound-2} and $\frac{1}{2}<t<\alpha<1$ with $\alpha=1-3\ell -(1+d)\delta$.  Observing that 
$$\left|\frac{q(v)}{p(v)}-1\right|\leq\E_*^{(n)}\left[\left|\frac{\tilde{q}(v,\beta_n^*)}{p(v)}-1\right|\cdot\one_{A_n}\right]+\P_*^{(n)}[A_n^c]\max\left\{\frac{1}{p(v)},1\right\}\leq (3C+3)n^{-t}+\P_*^{(n)}[A_n^c]\max\left\{\frac{1}{p(v)},1\right\}$$\label{prcondbth} where $\P_*^{(n)}:=\P[\cdot|\mathcal{T},\gamma_n]$, we see that to complete the proof we simply need to establish the following claim.

{\bf Claim:} With probability $1$, $\P_*^{(n)}[A_n^c]\max\left\{\frac{1}{p(v)},1\right\}>n^{-t}$ for only finitely many children of the backbone.

{\sc Proof of Claim}: We begin with the observation that
\begin{align}
\left|\frac{\frac{g(\TT(v),p_c +\epsilon)}{\sum g(\TT(w), p_c+\epsilon)}  }{p(v)} - 1 \right| 
&=   \left|\frac{\frac{W(v) + E(v,\epsilon)}{\sum W(w) + E(w,\epsilon)}} 
{\frac{W(v)}{\sum W(w)}} - 1 \right| = \left| \frac{E(v,\epsilon)\sum W(w) -W(v)\sum E(w,\epsilon)}{W(v)\sum[W(w) 
+ E(w,\epsilon)]} \right|.
\end{align}
Now using Corollary \ref{cor:E-bound-2}, we see that if we start with a fresh Galton-Watson tree then
\begin{align}
\left|\frac{\frac{g(\TT(v),p_c+n^{-\frac{t}{a}})}{\sum g(\TT(w), p_c+n^{-\frac{t}{a}})}  }{p(v)} - 1 \right| &\leq C 
\frac{W(v)n^{-t} \sum W(w) + \sum W(w) n^{-t} W(v)}{W(v)\sum W(w)(1 - C n^{-t} )} \nonumber \\
&= 2C\frac{n^{-t}}{1 - C n^{-t}} \label{eq:ratio-error}
\end{align}
for every child $v$ of the root, with probability at least $1-C n^{-\frac{t}{a}\delta'}$.  If we now condition on $h_n^*\leq n^{-\frac{t}{a}}$ and combine \eqref{eq:ratio-error} with Proposition \ref{pr:variance} we find that \begin{equation}\label{firstmjrbnd}\left|\frac{\frac{g(\TT(v),\beta_n^*)}{\sum g(\TT(w), \beta_n^*)}  }{p(v)} - 1 \right|\leq 2C\frac{n^{-t}}{1 - C n^{-t}}\end{equation}for every $v$ such that $\Par{v}=\gamma_n$ with probability at least $1-C n^{-\frac{t}{a}(1-\frac{1}{p})\delta'}$.

For the next step we recall that equation \eqref{eq:Aj-sandwich} tells us that \begin{equation}\label{bndfnldiffwpr}\frac{1}{p(v)}\left|\qt(v,\beta_n^*) - \frac{g(\TT(v),\beta_n^*)}{\sum\limits_{w : 
\Par{w} = \gamma_n} g(\TT(w), \beta_n^*)} \right| \leq \frac{g(\TT(\gamma_n), \beta_n^*)}{1 - 
g(\TT(\gamma_n), \beta_n^*)}\frac{\frac{g(\TT(v), \beta_n^*)}{\sum g(\TT(w),\beta_n^*)}}{p(v)}.\end{equation} Using \eqref{firstmjrbnd}, we see that when we condition on $h_n^*\leq n^{-\frac{t}{a}}$, the latter fraction in \eqref{bndfnldiffwpr} is bounded by, say $2$, for every child of $\gamma_n$, with probability at least $1-C n^{-\frac{t}{a}(1-\frac{1}{p})\delta'}$.  For the former fraction, we note that because $g(\TT(v),p_c + \ee) \leq C' \ee \overline{W}(v)$ for all $\ee$ bounded uniformly away from $1 - p_c$ (see Proposition \ref{pr:g-upper-bound}), it follows that for a fresh Galton-Watson tree and for $s$ bounded uniformly away from $0$, we have $$\P[g(\TT(v),p_c + n^{-s}) > n^{-t}] \leq  \P[C' \overline{W} > n^{s-t}] 
\lesssim n^{-p(s-t)}.$$Now setting $s=\frac{t}{a}$ and combining the above string of inequalities with Proposition \ref{pr:variance} we find that $$\P[g(\TT(\gamma_n),\beta_n^*) > n^{-t}|h_n^*\leq n^{-\frac{t}{a}}]\lesssim n^{-(p-1)(\frac{1}{a}-1)t}.$$Combining this with what we determined about the second fraction in \eqref{bndfnldiffwpr}, it now follows that if we condition on $h_n^*\leq n^{-\frac{t}{a}}$, then \begin{equation}\label{scndmjrbnd}\frac{1}{p(v)}\left|\qt(v,\beta_n^*) - \frac{g(\TT(v),\beta_n^*)}{\sum\limits_{w : 
\Par{w} = \gamma_n} g(\TT(w), \beta_n^*)} \right|\leq 3n^{-t}\end{equation} for every child of $\gamma_n$, with probability at least $1-C'' n^{-\frac{t}{a}(1-\frac{1}{p})\delta'}$ (where we're using the fact that $\frac{1}{a}(1-\frac{1}{p})\delta'\leq(p-1)(\frac{1}{a}-1)$).  Finally, putting \eqref{scndmjrbnd} together with \eqref{firstmjrbnd} and Theorem \ref{th:dual-piv-decay}, and defining $t':=\frac{t}{a}(1-\frac{1}{p})\delta'$, we get that $\E\left[\P_*^{(n)}[A_n^c]\right]$ is $O(n^{-t'})$.

 From this last result involving $\E\left[\P_*^{(n)}[A_n^c]\right]$, we know that for any constant $C_1$ such that $0<C_1<1$ we have \begin{equation}\label{bndonancp}\P\left[\P_*^{(n)}[A_n^c]>n^{-C_1 t'}\right]=O(n^{-(1-C_1)t'}).\end{equation} For the next step, we note that for any $t''>0$ and any constant $C_2$ with $0<C_2<1$, the probability $\frac{1}{p(v)}>n^{t''}$ for any child of $\gamma_n$, is bounded by \begin{equation}\label{bndwlngsm}\P[W(v)<\mu n^{-C_2 t''}\ \text{for at least 1 child of }\gamma_n]+\P[W(\gamma_n)\geq n^{(1-C_2)t''}],\end{equation} which is $O\left(\max\left\{n^{-(p-1)(1-C_2)t''},n^{-(1-\frac{1}{p})\frac{1}{q}C_2 t''}\right\}\right)$, as discussed at the end of section \ref{sec:bounds-E}.

To finish establishing the claim, we now need to show that $t,\delta,d,\ell,t'',\ \text{and }C_1$ can be chosen so that \begin{itemize}
    \item[(i)]$\frac{1}{2}<t<a<1$ 
    \item[(ii)]$\frac{1}{p(v)}>n^{t''}$ only finitely often with probability $1$
    \item[(iii)] $\P_*^{(n)}[A_n^c]>n^{-C_1 t'}$ only finitely often with probability $1$
    \item[(iv)] $C_1 t'-t''>t$, i.e. $n^{-C_1 t'}\cdot n^{t''}\leq n^{-t}$\,.
\end{itemize}
To accomplish this, we first note that it follows from \eqref{bndwlngsm} and the Borel-Cantelli lemma that the second condition will hold if $t''>\frac{1+pq}{p-1}$.  Hence, the fourth condition then reduces to $C_1 t'-\frac{1+pq}{p-1}>t$.  Combining this with the third condition, which by \eqref{bndonancp} and Borel-Cantelli will be satisfied if $(1-C_1)t'>1$, we find that proving our claim is reduced to finding $t,\delta,\ell, \text{and }d$ with $\frac{1}{2}<t<a<1$ such that $$t'>1+\frac{1+pq}{p-1}+t=\frac{p}{p-1}\left(1+q\right)+t.$$Using our formulas for $t'$ and $\delta'$, this can be written as\begin{equation}\label{pairofconds}\left(\frac{1}{a}\left(1-\frac{1}{p}\right)\min\left\{p\ell -\delta,\frac{d\delta}{q}\right\}-1\right)t>\frac{p}{p-1}\left(1+q\right)\end{equation}

It now suffices to show that \eqref{pairofconds} can be made to hold for $t=a=\frac{1}{2}$.  Substituting $\frac{1}{2}$ for $t$ and $a$ in \eqref{pairofconds} and noting that $a=\frac{1}{2}\implies\delta=\frac{1}{1+d}\left(\frac{1}{2}-3\ell\right)$, \eqref{pairofconds} becomes $$\Bigg(\left(1-\frac{1}{p}\right)\min\left\{p\ell -\frac{1}{1+d}\left(\frac{1}{2}-3\ell\right),\frac{d}{1+d}\cdot\frac{1}{q}\left(\frac{1}{2}-3\ell\right)\right\}-\frac{1}{2}\Bigg)>\frac{p}{p-1}(1+q).$$ Observing that the expression on the left is increasing with respect to $d$, we take $d\to\infty$, which gives$$\Bigg(\left(1-\frac{1}{p}\right)\min\left\{p\ell,\frac{1}{q}\left(\frac{1}{2}-3\ell\right)\right\}-\frac{1}{2}\Bigg)>\frac{p}{p-1}(1+q).$$
Expressing this as a pair of inequalities and then simplifying we get $$\frac{p}{(p-1)^2}(1+q)+\frac{1}{2(p-1)}<\ell<\frac{1}{3}\left(\frac{1}{2}-\frac{p^2 q(1+q)}{(p-1)^2}-\frac{pq}{2(p-1)}\right).$$For such an $\ell$ to exist it suffices to have $$\frac{p}{(p-1)^2}(1+q)+\frac{1}{2(p-1)}<\frac{1}{3}\left(\frac{1}{2}-\frac{p^2 q(1+q)}{(p-1)^2}-\frac{pq}{2(p-1)}\right).$$Now simplifying the above inequality, we get \eqref{lastineneeded}, thus completing the proof of the proposition. $\Cox$
\end{proof}

\bibliographystyle{alpha}
\bibliography{Bib}

\section*{Notation}

\subsubsection*{Trees}

 \begin{longtable}{p{1.5cm}p{11cm}p{1cm}} 
 $\rtt$ & root of tree \df & \pageref{rtt-def}\\
 $\ulam$ & canonical tree \df & \pageref{ulam-def} \\
 $\VU$ & canonical vertex set \df & \pageref{VU-def}\\
 $|v|$ & depth of a vertex \df & \pageref{depth-def} \\
 $\Par{v}$ & parent of a vertex \df & \pageref{Par-def} \\ 
 $v \sqcup i$ & $i$th child of $v$ \df & \pageref{sqcup-def}  \\
 $v \wedge w$ & least common ancestor of two nodes \df & \pageref{LCA-vertex-def} \\
 $\gamma \wedge \gamma'$ & least common vertex of two paths \df & \pageref{LCA-path-def} \\
 $\gamma_n$ & $n$th vertex in path $\gamma$ \df & \pageref{gamma_n-def} \\
 $\sigma^v$ & shift to vertex $v$ \df & \pageref{shift-def} \\
 $T(v)$ & subtree of $v$ rooted at $v$ \df & \pageref{T(v)-def} \\
 $T^*(v)$ & $T \setminus T(v)$ \df & \pageref{dual-tree-def} \\
 $T_n$ & set of nodes of $T$ at depth $n$ \df  & \pageref{T_n-def} \\
 $\partial T$ & set of infinity non-backtracking paths from root \df & \pageref{partial-T-def} \\
 $Z_n$ & number of descendants at height $n$ \df & \pageref{Z_n-def} \\
 $Z_n(v)$ & number of offspring of $v$ in generation $n + |v|$ \df & \pageref{Z_n(v)-def} \\
 $Z_n^{(i)}(v)$ & number of $n$th generation descendents of $v$ that 
	pass through $v \sqcup i$ \df & \pageref{Z_n^(i)-def} \\
 \end{longtable}

\subsubsection*{Branching processes}

 \begin{longtable}{p{1.5cm}p{11cm}p{1cm}}  
 $(\Omega,\F,\P)$ & the probability space \df &\pageref{space-def}\\
 $\phi$ & probability generating function for progeny distribution\df  &\pageref{phi-def}\\
  $Z$ & generic random variable with p.g.f. $\phi$ \df &\pageref{Z-def} \\
 $\mu$ &$=\E Z = \phi'(1)$ \df & \pageref{mu-def}\\
 $p_c$ &$= 1/\mu$ \df & \pageref{p_c-def} \\
 $\{ \deg_v \}$ & IID $\sim \phi$ variables that construct the Galton-Watson tree \df & \pageref{deg-def} \\
 $\T$ & $\sigma$-algebra generated by $\{ \deg_v \}$ \df &\pageref{T-def}\\
 $\GW$ & $= \P|_{\T}$, the Galton-Watson measure \df &\pageref{GW-def}\\
 $\TT$ & random rooted subtree chosen with Galton-Watson measure \df & \pageref{TT-def}\\
 $\E_*$ &$=\E[\cdot \| \T]$ \df &\pageref{E_*-def} \\
 $\P_*$ &$=\P[\cdot \| \T]$ \df & \pageref{P_*-def}\\
 $W_n$ & the martingale $\mu^{-n} Z_n$ \df & \pageref{W_n-def} \\
 $W_n(v)$ &$=\mu^{-n}Z_n(v)$ \df & \pageref{W_n(v)-def} \\
 $W$ &$=\lim W_n$ \df & \pageref{W-def}\\
 $\overline{W}$ &$= \max_n W_n$ \df & \pageref{Wbar-def} \\
 $W_n^{(i)}(v)$ &$= \mu^{-n} Z_n^{(i)}(v)$ \df & \pageref{W_n^(i)(v)-def}
\end{longtable}

\subsubsection*{Percolation}

\begin{longtable}{p{1.5cm}p{11cm}p{1cm}} 
 $U_v$ & uniform random variables defining percolation \df & \pageref{U_v-def}\\
 $\F_n$ & $\sigma$-algebra up to level $n$ \df &\pageref{F_n-def}\\
 $\F_n'$ & $\sigma$-algebra of pivots up to $n$ and entire tree \df &\pageref{F_n'-def} \\
 $I$ & invasion cluster \df &\pageref{I-def} \\ 
 $\gamma$ & the backbone of invasion percolation \df &\pageref{gamma-def} \\
 $v \leftrightarrow_p w$ & event that $v$ and $w$ connected in
	$p$ percolation \df & \pageref{conn-def} \\
 $H(p)$ & event that root is connected to infinity in 
	$p$ percolation \df & \pageref{H(p)-def} \\
 $g(T,p)$ & probability that root of $T$ is connected 
	to infinity in $p$ percolation \df & \pageref{g(T,p)-def}\\
 $g(p)$ &$\E g(\TT,p)$ \df & \pageref{g(p)-def} \\
 $K$ &$ \lim_{\ee \to 0^+} g(p_c + \ee) / \ee$\df  & \pageref{K-def} \\
 $\CC_v$ & $\sigma(T(v))$ \df & \pageref{CC_v-def}\\
 $\CC_n$ & $\sigma(T(\gamma_n))$ \df & \pageref{CC_n-def} \\
 $\BB_v^*$ & $\sigma(T \setminus T(v))$ \df & \pageref{BB_v^*-def} \\
 $\BB_n^*$ & $\sigma(T \setminus T(\gamma_n))$ \df & \pageref{BB_n^*-def}\\
 $\BB_n^+$ & $\sigma( \BB_n^* \cup \{\pivot_n\} )$ \df & \pageref{BB_n^+-def}\\
 $\G_n$ &$=\sigma(T(v) ; |v| = n)$ \df & \pageref{G_n-def} \\
 $\pivot (v)$ & pivot of $v$ \df & \pageref{pivot-def}\\ 
 $\pivot_n$ & pivot of $n$th vertex of backbone \df &  \pageref{pivot_n-def}\\
 $\dual_v$ & dual pivot of $v$ \df & \pageref{dual-def} \\ 
 $\dual_n$ & dual pivot of $n$th vertex of backbone \df & \pageref{dual-n-def} \\
 $\E_*^{(n)}$ & $=\E[\cdot|\mathcal{T},\gamma_n]$ \df & \pageref{excondbth} \\
 $\P_*^{(n)}$ & $=\P[\cdot|\mathcal{T},\gamma_n]$ \df & \pageref{prcondbth}
\end{longtable}

\subsubsection*{Random measures}

\begin{longtable}{p{1.5cm}p{11cm}p{1cm}} 
 $\mu_T^n$ & uniform measure on $T_n$ \df & \pageref{mu_T^n-def} \\
 $\mu_T$ & limit uniform measure on $\partial T$ \df & \pageref{mu-T-def} \\
 $\nu_T$ & invasion measure $\partial T$ \df & \pageref{nu_T-def} \\
 $M_n$ & $= \frac{d\nu_T}{d\mu_T}\Big|_{\G_n}$ \df & \pageref{M_n-def} \\
 $p(v)$ & probability that limit uniform measure splits from 
	$\Par{v}$ to $v$ \df & \pageref{p(v)-def} \\
 $q(v)$ & probability that invasion measure splits from $\Par{v}$ to $v$ \df & \pageref{q(v)-def} \\
 $\qt(v,p)$ & probability that invasion measure splits from $\Par{v}$ to $v$, 
	given $\pivot(\Par{v}) \leq p$\df  & \pageref{qt-def} \\
 $X(v)$ & $=\sum_w q(w) \log[ q(w) / p(w]$ \df & \pageref{X-def}
 \end{longtable}

\end{document}